\documentclass[final,leqno,onefignum]{siamltex}
\usepackage{epsfig,amsfonts,latexsym,amsmath,amssymb}
\usepackage{graphicx}
\usepackage{subfig}
\usepackage{hyperref}
\usepackage{floatrow}
\usepackage{booktabs}
\usepackage{algorithm2e}

\title{On the construction of virtual interior point source travel
  time distances from the hyperbolic Neumann-to-Dirichlet map}

\author{Maarten. de  Hoop\ \footnotemark[2]
\and Paul Kepley\ \footnotemark[3]
\and Lauri Oksanen \footnotemark[4]}

\newcommand{\norm}[1]{\left\|#1 \right\|}

\def\Cut{\mathcal C} 
\def\R{\mathbb{R}}

\def\p{\partial}

\DeclareMathOperator*{\argmin}{argmin} 
\DeclareMathOperator{\divergence}{div_g} 
\DeclareMathOperator{\gradient}{grad_g} 
\DeclareMathOperator{\diam}{diam}
\DeclareMathOperator{\dist}{dist}

\DeclareMathOperator{\linearSpan}{span}
\DeclareMathOperator{\WF}{WF}

\DeclareMathOperator{\Vol}{Vol_{\mu}}
\def\meas{m}

\def\boundaryConds{N_\mu}
\def\dV{dV_{\mu}}
\def\dS{dS_g}
\def\dtdS{dt\otimes\dS}

\DeclareMathOperator{\wavecap}{cap}

\newtheorem{assumption}{Assumption}

\SetKwInput{kwLet}{let}
\SetKwInput{kwCompute}{compute}
\SetKwInput{kwSolve}{solve}

\SetStartEndCondition{ }{}{}%
\SetKwProg{Fn}{def}{\string:}{}
\SetKwFunction{Range}{range}
\SetKw{KwTo}{in}\SetKwFor{For}{for}{\string:}{}%
\SetKwIF{If}{ElseIf}{Else}{if}{:}{elif}{else:}{}%
\SetKwFor{While}{while}{:}{fintq}%

\begin{document}

\maketitle

\renewcommand{\thefootnote}{\fnsymbol{footnote}} 

\footnotetext[1]{This work was initiated during a visit by the first
  and third authors to the Institut Mittag-Leffler (Djursholm,
  Sweden). The work of the first and second authors was supported in
  part by the members of the Geo-Mathematical Imaging Group.}

\footnotetext[2]{Department of Computational and Applied Mathematics,
  Rice University, Houston, TX 77005 ({mdehoop@rice.edu}).}

\footnotetext[3]{Department of Mathematics, Purdue University, West
  Lafayette, IN 47907 ({pkepley@purdue.edu}).}

\footnotetext[4]{Department of Mathematics, University College London,
  Gower Street, London WC1E 6BT, UK ({l.oksanen@ucl.ac.uk}). The work
  of this author was partially supported by the Engineering and
  Physical Sciences Research Council (UK) grant EP/L026473/1.}
\renewcommand{\thefootnote}{\arabic{footnote}}

\pagestyle{myheadings} 

\thispagestyle{plain} 

\markboth{MAARTEN DE HOOP, PAUL KEPLEY, AND LAURI
  OKSANEN}{CONSTRUCTING TRAVEL TIME DISTANCES}

\begin{abstract}
We introduce a new algorithm to construct travel time distances
between a point in the interior of a Riemannian manifold and points on
the boundary of the manifold, and describe a numerical implementation
of the algorithm. It is known that the travel time distances for all
interior points determine the Riemannian manifold in a stable manner.
We do not assume that there are sources or receivers in the interior,
and use the hyperbolic Neumann-to-Dirichlet map, or its restriction,
as our data. Our algorithm is a variant of the Boundary Control
method, and to our knowledge, this is the first numerical
implementation of the method in a geometric setting.
\end{abstract}



\section{Introduction}

We consider the inverse boundary value problem for the acoustic wave
equation on a domain or on a Riemannian manifold $M$. The problem is
to recover a spatially varying sound speed $c$ inside the domain given
the corresponding Neumann-to-Dirichlet map, or its restriction on a
part of the boundary of the domain, say $\Gamma \subset \p M$.

In the acoustic case, waves propagate in a domain $M$ with speed $c$,
and the travel time of a wave between a pair of points $x$ and $y$ in
$M$ is given by the Riemannian distance $d(x,y)$ when computed with
respect to the travel time metric $c^{-2}dx^2$.  For this reason, it
is natural to formulate the inverse boundary value problem by using
concepts from Riemannian geometry. This also allows us to consider
anisotropic sound speeds in a unified way.

We present a new method to use the local restriction of the
Neumann-to-Dirichlet map to determine travel time distances of the
form $d(x,y)$ where $x$ is a point in the interior of the domain and
$y \in \Gamma$, that is, $y$ is a point in the set where we have
boundary measurements. We refer to these travel times as \emph{point
  source travel time data}, since the distance $d(x,y)$ corresponds to
the first arrival travel time from a (virtual) interior point source
located at $x$ as recorded at the boundary at $y$. We emphasize that
our method synthesizes the travel times from a point source in the
interior of $M$ without requiring an actual receiver or source at that
location.

Using the travel time distances we can introduce Dirac boundary
sources which generate waves the singularities of which focus in a
point. Such boundary sources have been referred to as focusing
functions in reflection seismology \cite{Wapenaar2014}.

To motivate our results, we note that travel times have previously
been shown to determine a Riemannian manifold with boundary see
e.g. \cite{Katchalov2001}. In particular, in the full boundary data
case, it has been shown that this determination is even stable
\cite{Katsuda2007}. Furthermore, it can be shown that travel times
determine shape operators that appear as data for the generalized Dix
method \cite{deHoop2014}. This is of particular interest in the
isotropic case, since that method allows for the local nonlinear
reconstruction of a sound-speed near geodesic rays.

Our method to determine boundary distances works by first using a
variant of the Boundary Control method (BC method) to determine
volumes of subdomains of $M$ referred to as wave caps. The volumes are
computed by solving a collection of regularized ill-posed linear
problems. The BC method originates from \cite{Belishev1987}, where it
was used to solve the inverse boundary problem for the acoustic
wave equation. We note that \cite{Belishev1992} was the first variant
of the BC method posed in a geometric setting. For a thorough overview
of the BC method, we refer to \cite{Katchalov2001} and
\cite{Belishev1997}. Regularization theory was first combined with the
BC method in \cite{Bingham2008}, and the procedure that we use to
determine volumes was developed in \cite{Oksanen2011}.

In the present paper, we introduce a procedure to construct point
source travel time data from the volumes of wave caps, and hence
reduce the inverse boundary value problem to the stable problem of
determining the Riemannian manifold $(M,g)$ given the point source
travel time data. In particular, our procedure splits the inverse
boundary value problem to an ill-posed but linear step (the volume
computation) and a non-linear but well-posed step (distance estimation
and reconstruction of the manifolds).

We describe a computational implementation of our method, and provide
a numerical example to demonstrate our technique. We remark that our
numerical example provides the first computational realization of a
geometric variant of the BC method. Moreover, we explain how the
instability of the volume computation step is manifest in our
numerical examples. All variants of the BC method that use partial
data contain similar unstable steps, and we believe that the
instability of the method reflects the ill-posedness of the inverse
boundary value problem itself. However, contrary to Calder{\'o}n's
problem, that is, the elliptic inverse boundary value problem
\cite{Calder'on1980, Sylvester1987}, it is an open question whether
the inverse boundary value problem for the wave equation is ill-posed
in general.  For results concerning the stability of Calder{\'o}n's
problem we refer to \cite{Alessandrini1988, Mandache2001}.

Also, we point out that under favorable geometric assumptions, the
hyperbolic inverse boundary value problem is known to have better
stability properties than Calder{\'o}n's problem, see
e.g. \cite{Bellassoued2011, Liu2012, Montalto2014, Stefanov1998,
  Stefanov2005} and references therein. However, these geometric
assumptions do not apply to the partial data problem that we consider
here.

\section{Statement of the Results}
\label{sec:1}
In this section we describe our data and assumptions, and what we
intend to recover from the data.

\subsection{Direct problem as a model for measurements}
We work in the following setting,

\begin{assumption}
$M$ is a smooth compact connected manifold with smooth boundary $\p
  M$, and $g$ is an unknown smooth Riemannian metric on $M$. 
\end{assumption}

Let $\mu \in C^\infty(\overline{M})$ be a strictly positive weight
function, note that we do not require any additional information about
$\mu$.  We consider the following wave equation on $(M,g)$ with
boundary source $f \in C_0^{\infty}((0,\infty) \times \p M)$,
\begin{equation}
  \label{eqn:ForwardProb}
  \begin{array}{lr}
    \partial_t^2 u(t,x) - \Delta_{g,\mu} u(t,x) = 0, & (t,x) \in (0,\infty) \times M,\\ 
    \boundaryConds u(t,x) = f(t,x), & (t,x) \in (0,\infty) \times \Gamma \\
    u(0,\cdot) = 0 ,\quad \partial_t u(0,\cdot) = 0, & x \in M.
  \end{array}
\end{equation}
The operator $\Delta_{g,\mu}$ is the weighted Laplace-Beltrami
operator, given by,
\begin{align*}
  \Delta_{g,\mu}w(x) := \mu^{-1} \divergence(\mu \gradient(w)),
\end{align*}
and, $\boundaryConds$ is the associated Neumann derivative,
\begin{align*}
  \boundaryConds w := -\mu \left\langle \nu, \gradient(w)\right\rangle_g,
\end{align*}
where $\nu$ is the inward pointing unit normal vector to $\p M$ in the
metric $g$, $\divergence$ and $\gradient$ are respectively the
divergence and gradient on $(M,g)$, and $\langle \cdot, \cdot
\rangle_g$ denotes the inner product with respect to the metric
$g$. We remark that $\Delta_{g,\mu}$ is defined so that we can also
handle the usual acoustic wave equation. Indeed, if we consider a
domain $M \subset \R^n$ along with a strictly positive function $c \in
C^\infty(\overline{M})$, then with respect to the conformally
Euclidean metric $g = c^{-2}dx^2$, the weight $\mu = c^{n-2}$ yields
$\Delta_{g,\mu} = c^2 \Delta$.

We now introduce our data model. We denote the solution to
(\ref{eqn:ForwardProb}) with Neumann boundary source $f$ by
$u^f(t,x)$. For $T > 0$, and $\Gamma \subset \p M$ an open set, we
define the local Neumann-to-Dirichlet operator associated with
(\ref{eqn:ForwardProb}) by,
\begin{align*}
  \Lambda_\Gamma^{2T} : f \mapsto u^f|_{(0,2T) \times \Gamma},
  \quad f \in C_0^\infty((0,2T) \times \Gamma).
\end{align*} 
The map $\Lambda_\Gamma^{2T}$ extends to a bounded operator on
$L^2((0,2T)\times \Gamma)$, see for instance \cite{Lasiecka1990}. For
data, we suppose that $\Lambda_{\Gamma}^{2T}$ is known.  We note that
$\Lambda_{\Gamma}^{2T}$ models boundary measurements for waves
generated with acoustic sources and receivers located on $\Gamma$,
where the waves are both generated and recorded on $\Gamma$ for $2T$
units of time.
In addition, we note that in seismic applications the
Neumman-to-Dirichlet map appears in simultaneous source acquisition.

Our primary interest is in constructing distances with respect to the
Riemannian metric $g$, and we denote the Riemannian distance between
the points $x,y \in M$ by $d(x,y)$. To simplify our distance
computation procedure, we assume:

\begin{assumption}
  \label{assuption:distances}
  The distances $d(y,z)$ are known for $y,z \in \Gamma$ with $d(y,z) <
  T$.
\end{assumption}

\noindent We note that this is not a major limitation since for $y,z
\in \Gamma$ with $d(y,z) < T$, the data $\Lambda_\Gamma^{2T}$
determines $d(y,z)$, see e.g. \cite[Section 2.2]{Dahl1}.

\subsection{Reconstruction of the point source travel time
  data}

We define $R(M)$ to be the set of \emph{boundary distance functions}
on $M$,
\begin{equation}
	R(M) = \{ r_x : x \in M, \text{ and for $z \in \p M$, $r_x(z)
          := d(x,z)$}\}.
\end{equation}
We note that, for $x\in M$ and $z \in \p M$, $r_x(z)$ gives the
minimum travel time from $x$ to $z$. With this interpretation,
$r_x(z)$ represents the first arrival time at $z$ from a wave
generated by a point source located at $x$.

In Section~\ref{section:SolvingForDistances} we develop a method to
synthesize values of $r_x$ from $\Lambda_\Gamma^{2T}$ for points $x$
indexed by a set of coordinates known as semi-geodesic
coordinates\footnote{Such coordinates are considered in seismology,
  where they are referred to as image ray coordinates
  \cite{Hubral1980}.}.  We refer to this procedure as forming
\emph{point source travel time data}, since our procedure reproduces
the travel time information for a point source located at $x$ without
having a source or receiver there.

The geometry of the supports of solutions to (\ref{eqn:ForwardProb})
inform our constructions. To be explicit, let $\tau$ be a function on
$\Gamma$ satisfying $0 \leq \tau(z) \leq T$ for all $z \in \Gamma$ and
define the \emph{domain of influence} of $\tau$,
$$   M(\tau) := \left\{ x \in M\ : \ \text{there exists a $z \in
     \bar \Gamma$ such that $d(x,z) \le \tau(z)$} \right\} .$$
We depict $M(\tau)$ in Figure \ref{subflt:domInfl}. Consider the set
\begin{equation}
  \label{eqn:Define_S_tau}
  S_\tau := \{(t, z) \in (0,T) \times \bar \Gamma : \ t \in (T -
    \tau(z), T) \}.
\end{equation} 
We recall that solutions to (\ref{eqn:ForwardProb}) exhibit finite
speed of propagation in the metric $g$, and specifically, if
$\supp(f) \subset S_\tau$ then $\supp(u^f(T,\cdot)) \subset M(\tau)$.

When $\tau$ is a multiple of an indicator function, we will
occasionally use a special notation for $M(\tau)$. To be specific, we
denote the indicator function of a set $S$ by $1_S$, and for $s \geq
0$ we will use the notation $M(\Gamma,s) := M(s1_\Gamma)$ and $M(y,s)
:= M(s1_{\{y\}})$.

We denote the unit sphere bundle $SM := \{\xi \in TM :\ |\xi|_g =
1\}$, and define the inward/outward pointing sphere bundles by $\p_\pm
S M := \{\xi \in \p SM :\ (\xi, \pm \nu)_g > 0 \}$, where $\nu$ is the
inner unit normal vector field on $\p M$.  We define the {\em exit
  time} for $(x, \xi) \in SM \setminus \overline{\p_+ S M}$, by
\begin{align*}
\tau_M(x, \xi) := \inf \{ s \in (0, \infty):\ \gamma(s; x, \xi) \in \p
M \},
\end{align*}
where $\gamma(\cdot; x, \xi)$ is the geodesic with the initial data
$\gamma(0) = x$, $\dot \gamma(0) = \xi$. 

For $y \in \Gamma$ we define $\sigma_{\Gamma}(y)$ to be the maximal
arc length for which the normal geodesic beginning at $y$ minimizes
the distance to $\Gamma$. That is,
  \begin{align*}
    \sigma_{\Gamma}(y) &:= \max \{ s \in (0, \tau_M(y, \nu)]:\ d(\gamma(s; y,
      \nu), \Gamma) = s\}.
  \end{align*}
We recall, see e.g. \cite[p. 50]{Katchalov2001} that
$\sigma_{\Gamma}(y) > 0$ for $y \in \Gamma$. Moreover, $\sigma_\Gamma$
is lower semi-continuous, see e.g. \cite[Lemma 12]{Lassas2014}. We
define
\begin{equation}
  x(y,s) := \gamma(s;y,\nu) \text{\quad for $y \in \Gamma$ and $0 \leq
    s < \sigma_{\Gamma}(y)$.}
\end{equation}
The mapping $(y,s) \mapsto x(y,s)$ is a diffeomorphism from $\{(y,s) :
y \in \Gamma, 0 \leq s < \sigma_\Gamma(y) \}$ onto its image, so we
will refer to $(y,s)$ as semi-geodesic coordinates for $x(y,s)$.  This
is a slight abuse of terminology, since the pair $(y,s)$ belongs to
$\Gamma \times [0,\infty)$ instead of a subset of $\R^n$. On the other
  hand, by selecting local coordinates on $\Gamma$ these
  ``coordinates'' can be made into legitimate coordinates.

Next, we recall the definition of the \emph{cut locus} of $\Gamma$,
\begin{align*}
  \Cut &= \left\{x\left(y,\sigma_{\Gamma}(y)\right) :\ y \in \Gamma\right\}.
\end{align*}
We depict $\Cut$ in Figure \ref{subflt:cutLocus}. Due to the
lower semi-continuity of $\sigma_\Gamma$ and the boundedness of
$\Gamma$, one sees that the distance between $\Gamma$ and $\Cut$ is
positive.

\begin{figure}[h!]
  \centering 
  
  \subfloat[]{
    \includegraphics[height = 1.4in]{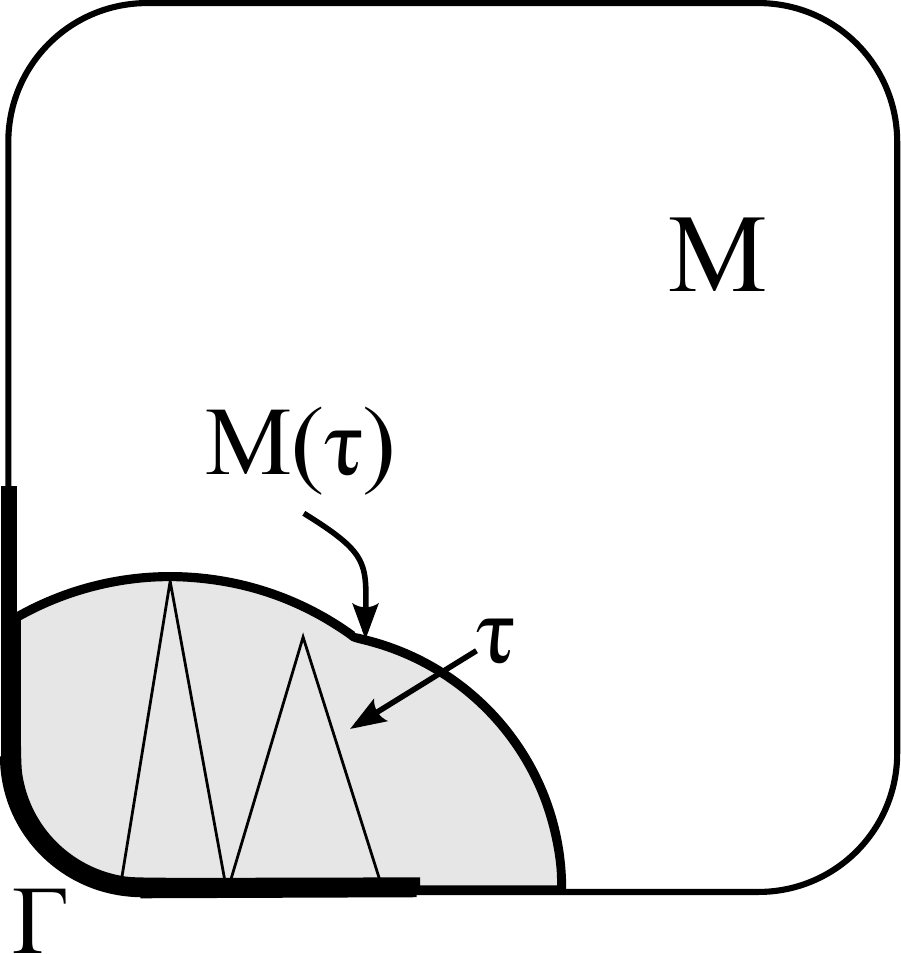}
    \label{subflt:domInfl}
  }
  \qquad \qquad \qquad \qquad
  \subfloat[]{    
    \includegraphics[height = 1.4in]{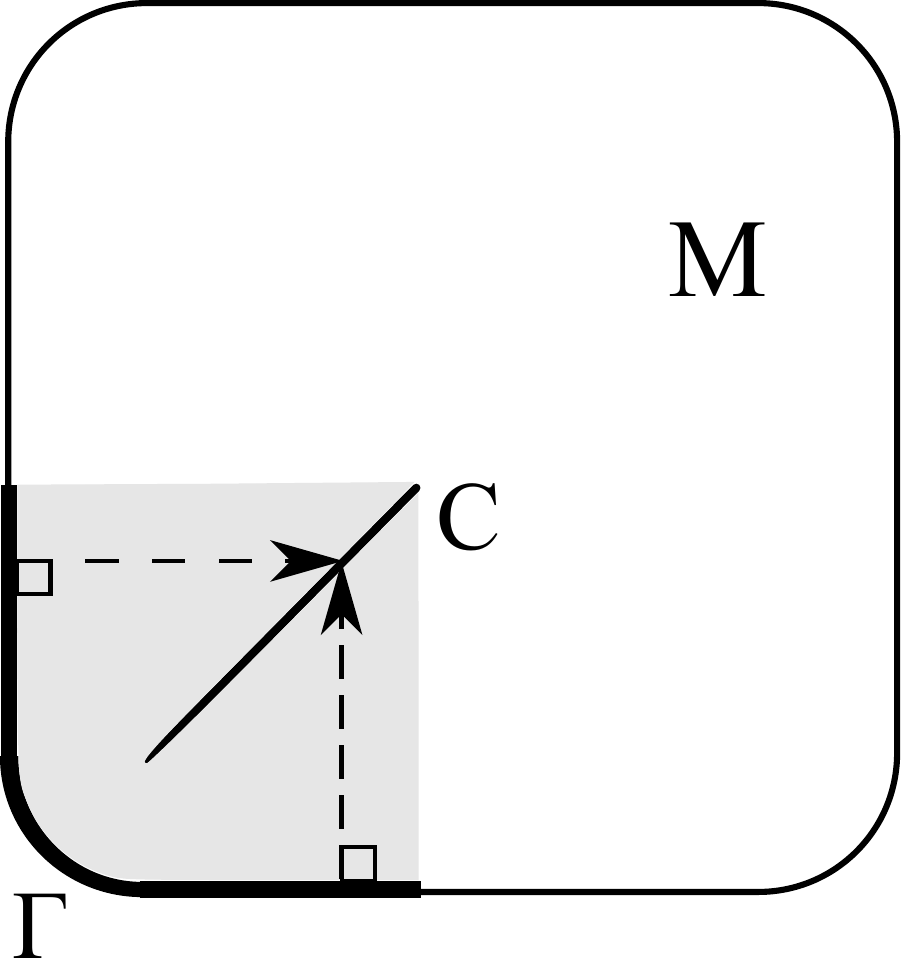}
    \label{subflt:cutLocus}
  }

  \caption{(a) The domain of influence for a $\tau$ in
    $C(\overline{\Gamma})$ along with the profile of $\tau$.  (b) The
    cut locus of $\Gamma$ along with a pair of equal length geodesics
    showing the break-down of the semi-geodesic coordinates at
    $\Cut$. The shaded region is the subset of $M$ that supports
    semi-geodesic coordinates.}
  
\end{figure}

We will use the following notions of volume: let $dV_g$ and $\dS$
denote the Riemannian volume densities of $(M,g)$ and $(\p M, g|_{\p
  M})$ respectively. We remark that $\dS$ is determined on $\Gamma$ by
$\Lambda_\Gamma^{2T}$, see e.g. \cite[Section 2.2]{Dahl1} so we assume
that it is known. We define the natural Riemannian volume density
associated with $\Delta_{g,\mu}$ by $\dV := \mu dV_g$. We remark that
its name derives from the fact that $\Delta_{g,\mu}$ is self-adjoint
on $L^2(M;\dV)$ with domain $H_0^1(M) \cap H^2(M)$. The volume density
$\dV$ determines a volume measure which we denote $\Vol$. In addition,
we will use the following shorthand notation for volumes of domains of
influence $m(\tau) := \Vol(M(\tau))$.

We now describe a set of geometrically relevant subsets whose volumes
will allow us to determine distances. Let $y \in \Gamma$, and let
$s,h > 0$ satisfy $s + h < \sigma_{\Gamma}(y)$. We define the
\emph{wave cap},
\begin{align*}
  \wavecap_\Gamma(y,s,h) := M(y,s+h) \setminus M^\circ(\Gamma,s),
\end{align*}
where $M^\circ(\Gamma,s) = \{ x \in M : d(x,\Gamma) < s\}$.  Note that
under the above hypotheses, $x(y,s)$ belongs to
$\wavecap_\Gamma(y,s,h)$. We will use volumes of wave caps to
determine distances.

Our main result is an algorithm to use the data
$\Lambda_{\Gamma}^{2T}$ to construct distances of the form
$r_{x(y,s)}(z)$ for $y,z \in \Gamma$ and $s > 0$ with $d(x(y,s),z) <
\min(\sigma_\Gamma(y),T)$. Our procedure can also be viewed as a
constructive proof of the following known result, see
e.g. \cite{Katchalov2001}:

\begin{theorem}
  \label{theorem:VolumesDetermineDistances}
  Let $y,z \in \Gamma$ and $s > 0$ with
  $d(x(y,s),z) < \min(\sigma_\Gamma(y),T)$. Then
  $\Lambda_{\Gamma}^{2T}$ determines $r_{x(y,s)}(z)$.  
\end{theorem}

The constructive proof will be given in
Section~\ref{sec:ConstructingDistances}. We note that this construction can
also be viewed as a series of experiments. Following the proofs in
Section~\ref{sec:ContainsAlgorithm}, we provide an algorithmic overview of our
distance computation procedure.

\section{The Boundary Control method}
\label{section:SolvingForDistances}

In this section we describe the elements of the BC method required to
determine $m(\tau)$ from $\Lambda_{\Gamma}^{2T}$. In addition, we
briefly contrast our technique to alternative approaches to the BC
method, and provide an overview of some computational aspects of the
BC method.

The purpose of the BC method is to gain information about the interior
of $M$ by processing boundary measurements for waves that propagate in
$M$. To begin, we recall the \emph{control map}, $W_\tau$, which takes
a Neumann boundary source $f$ to the corresponding solution at time
$T$. That is, let $\tau : \overline{\Gamma} \rightarrow [0,T]$ be
continuous or a step function with open level sets and define,
\begin{equation*} 
  W_{\tau}f := u^f(T,\cdot), \quad W_\tau : L^2(S_\tau) \rightarrow L^2(M).
\end{equation*} 
The map $W_\tau$ is continuous, compact even, as a map from
$L^2(S_\tau)$ into $L^2(M)$, see e.g. \cite{Lasiecka1990,
  Tataru1998}. We will write $W := W_\tau$ in the special case that
$\tau \equiv T$, and we note that in this case, $S_\tau = (0,T) \times
\Gamma$. Thus, for any other $\tau : \Gamma \rightarrow [0,T]$,
$W_{\tau}$ can be viewed as a restriction of $W$ to sources supported
in $S_\tau$.

One cannot directly observe the output of $W_\tau$ from boundary
measurements because its output is a wave in the interior of
$M$. Thus, in order to deduce information about the interior of $M$,
one forms the \emph{connecting operator},
\begin{equation*}
  K_\tau := W_\tau^* W_\tau, \quad K : L^2(S_\tau) \rightarrow
  L^2(S_\tau).
\end{equation*} 
The continuity of $W_\tau$ implies that $K_\tau$ is a continuous
operator on $L^2(S_\tau)$. The practical utility of $K_\tau$ is that
it can be computed by processing the boundary data,
$\Lambda_\Gamma^{2T}$, see (\ref{eq_connecting}) below. This fact was
first observed by Blagoveschenskii in the \mbox{$1$+$1$-dimensional} case
\cite{Blago1971}. We remark that $K_\tau$ derives its name from the
fact that it connects the inner product on boundary sources with the
inner product on waves in the interior. That is, for $f$, $h$ in
$C_0^\infty(S_\tau)$,
\begin{align}
  \label{eq_Blago_identity}
  (u^f(T,\cdot), u^h(T,\cdot))_{L^2(M; \dV)} &= (f, K
  h)_{L^2(S_\tau;\,\dtdS)}.
\end{align}

We next recall the ``Blagoveschenskii Identity,'' which gives an
expression for $K_\tau$ in terms of the data
$\Lambda_{\Gamma}^{2T}$. In particular, we use the expression for $K_\tau$
appearing in \cite{Oksanen2013},
\begin{align}
  \label{eq_connecting}
  &K_\tau = P_\tau \left(J\Lambda^{2T}_{\Gamma} \Theta - R
  \Lambda^{T}_{\Gamma} R J \Theta \right)P_\tau.
\end{align}
Here, $\Theta : L^2((0,T) \times \Gamma) \rightarrow L^2((0,2T) \times
\Gamma)$ is the inclusion (zero padding) given by:
\begin{equation*}
  \Theta f(t,\cdot) := \left\{
  \begin{array}{cl}
    f(t,\cdot) & 0 < t \leq T, \\
    0 & T < t < 2T,
  \end{array}
  \right.
\end{equation*} 
$R : L^2((0,T) \times \Gamma) \rightarrow L^2((0,T) \times \Gamma)$ is
the time reversal on $(0,T)$ given by:
\begin{equation*}
  Rf(t,\cdot) := f(T - t,\cdot) \quad 0 < t < T,
\end{equation*} 
 $J: L^2((0,2T) \times \Gamma) \rightarrow L^2((0,T) \times \Gamma)$
is the time integration, given by:
\begin{align*}
  &Jf(t,\cdot) := \frac{1}{2} \int_t^{2T - t} f(s,\cdot) \,ds \quad 0 < t < T,
\end{align*}
and $P_\tau : L^2((0,T) \times \Gamma) \rightarrow L^2((0,T) \times
\Gamma)$ is the orthogonal projection onto $L^2(S_\tau)$ given by:
\begin{equation}
  \label{eqn:pTau}
  P_\tau f := 1_{S_\tau} \cdot f.
\end{equation}
We will use the special notation $K := K_\tau$ when $\tau \equiv
T$. In this case, the operator $P_\tau$ coincides with the identity
and (\ref{eq_Blago_identity}) can be written as $K =
J\Lambda_\Gamma^{2T} \Theta - R \Lambda_\Gamma^T R J \Theta$. Thus,
for any $\tau$, the operator $K_\tau$ can be expressed as $K_\tau =
P_\tau K P_\tau$.

\subsection{Overview of BC method variants}
\label{sec:BlagoTypeIdents}

There are several variants of the BC method, all of which are based on
solving control problems of the form: Given a function $\phi$ on $M$,
and a function $\tau : \overline{\Gamma} \rightarrow [0,T]$, find a
boundary source $f$ such that
\begin{align}
\label{eq_BC_control}
W_\tau f = \phi.
\end{align}
In general, this problem is not solvable since the range of $W_\tau$
is generally not closed. On the other hand, it can be shown that
\emph{approximate controllability} holds, that is, there is a sequence
$(f_j)_{j=1}^\infty \subset C_0^\infty(S_\tau)$ such that
\begin{align}
  \label{eq_BC_lim_u}
  \lim_{j \to \infty} W_\tau f_j = 1_{M(\tau)} \phi,
  \quad \text{in $L^2(M)$}.
\end{align}
The approximate controllability follows from 
the hyperbolic unique continuation result by Tataru \cite{Tataru1995}
by a duality argument, see e.g. \cite[p. 157]{Katchalov2001}.

The original version of the BC method \cite{Belishev1987} uses the
Gram-Schmidt orthonormalization to find a sequence
$(f_j)_{j=1}^\infty$ satisfying (\ref{eq_BC_lim_u}).  The method was
implemented numerically in \cite{Belishev1999}, and it requires
choosing an initial system of boundary sources, see step 2 in
\cite[p. 233]{Belishev1999}.  No constructive way to choose the
initial boundary sources is given, and some choices may lead to an
ill-conditioned orthonormalization process, see the discussion in
\cite{Bingham2008}.

More recently, Bingham, Kurylev, Lassas and Siltanen introduced a
variant of the BC method where the Gram-Schmidt process is replaced by
a quadratic optimization \cite{Bingham2008}.  Their method is posed in
the case $\Gamma = \p M$, and is based on constructing a sequence
$(f_j)_{j=1}^\infty$ such that the limit (\ref{eq_BC_lim_u}) becomes
focused near a point.  To elaborate, their method considers an
arbitrary $h \in L^2((0,T) \times \p M)$ with $\phi$ chosen as $\phi =
Wh$.  For a point $y \in \p M$ and small enough $0 < s,r < T$, they
choose appropriate $\tau$ to produce a sequence of sources
$(f_j)_{j=1}^\infty \subset S_\tau$ such that $W f_j \rightarrow
1_{\wavecap_{\p M}(y,s,r)} Wh$.  However, no constructive procedure to
choose the boundary source $h$ is given, and some choices may lead to
sequences such that this limit vanishes also near the point where it
should be focused, see the assumption on the non-vanishing limit in
\cite[Corollary 2]{Bingham2008}.  We note that the method
\cite{Bingham2008} has not been implemented numerically.

Our approach employs a quadratic optimization similar to
\cite{Bingham2008} but differs from it by selecting $\phi = 1$ in
place of $W h$. By solving the approximate control problems for this
choice of $\phi$, we can compute volumes $m(\tau)$ for certain
functions $\tau : \overline{\Gamma} \to [0,T]$.  We note that the
method we use to compute these volumes was developed in
\cite{Oksanen2011, Oksanen2013a}, and it was applied to an inverse
obstacle problem in \cite{Oksanen2013}.  Here we show how to compute
the boundary distance functions from the volumes $m(\tau)$.

Our method contains only constructive choices of boundary sources, and
it allows us to understand the numerical errors that we make in each
step of the algorithm. In Section~\ref{sec:Numerics}, we see that the
dominating source of error in our numerical examples is related to the
instability of the control problem (\ref{eq_BC_control}) under the
constraint $\supp(f) \subset S_\tau$.  This instability is inherent in
all the variants of the BC method mentioned above.

In addition to \cite{Belishev1999}, the only multidimensional
implementation of a variant of the BC method, that we are aware of, is
\cite{Pestov2010}.  This variant is based on solving the control
problem (\ref{eq_BC_control}) without the constraint
$\supp(f) \subset S_\tau$.  The target function $\phi$ is chosen to be
harmonic, and the method exploits the density of products of harmonic
functions in $L^2(M)$.  Such an approach works only in the isotropic
case, that is, in the case of the wave equation
$\p_t^2 - c(x)^2 \Delta$ where the sound speed $c(x) > 0$ is scalar
valued.

We also mention that the original version of the BC method
\cite{Belishev1987} assumes the wave equation to be isotropic, and
that in \cite{Liu2012}, an approach similar to \cite{Pestov2010} was
shown to recover a lowpass version of the sound speed in a Lipschitz
stable manner under additional geometric assumptions.  Furthermore, we
refer to \cite{Kabanikhin2005} for a comparison of the BC method and
other inversion methods in the $1$+$1$-dimensional case.

\subsection{Regularized estimates of volumes of domains of influence}
\label{subsection:RegEstOfVol}

We now explain how we pose our approximate control problems, and how
we use their solutions to compute volumes of domains of influence. To
begin, let $\tau : \overline{\Gamma} \rightarrow [0,T]$ be either a
step function with open level sets or $\tau \in
C(\overline{\Gamma})$. We obtain an approximate solution to
(\ref{eq_BC_control}) with right-hand side $\phi = 1$, by solving the
following minimization problem: for $\alpha > 0$ let
\begin{equation}
  \label{eqn:argminProb}
  f_\alpha := \argmin_{f \in L^2(S_\tau)} \|u^f(T,\cdot) -
  1\|_{L^2(M;\dV)}^2 + \alpha \|f\|_{L^2(S_\tau;\dtdS)}^2.
\end{equation}
As was shown in \cite{Oksanen2011}, for $\tau$ as above: this problem
is solvable, the solution can be obtained by solving a linear problem
involving $K_\tau$, and $u^{f_\alpha}(T,\cdot) \rightarrow
1_{M(\tau)}$ as $\alpha \rightarrow 0$. For the convenience of the
reader, we outline the proof here, and moreover, we recall that the
approximate control solutions, $f_\alpha$, can be used to compute
$m(\tau)$.

To show that (\ref{eqn:argminProb}) has a solution we first recall two
results about Tikhonov regularization.  For proofs see
e.g. \cite[Th. 2.11]{Kirsch2011} and \cite{Oksanen2013}, respectively.

\begin{lemma}
  \label{lemma:TikhExist}
  Suppose that $X$ and $Y$ are Hilbert spaces. Let $y \in Y$ and let $A
  : X \to Y$ be a bounded linear operator.  Then for all $\alpha > 0$
  there is a unique minimizer of
  \begin{align*}
    \norm{A x - y}^2 + \alpha \norm{x}^2
  \end{align*}
  given by $x_\alpha = (A^* A + \alpha)^{-1} A^* y$.
\end{lemma}

\begin{lemma}
  \label{lemma:TikhConv}
  Suppose that $X$ and $Y$ are Hilbert spaces.  Let $y \in Y$ and let
  $A : X \to Y$ be a bounded linear operator with range $R(A)$.  Then
  $A x_\alpha \to Qy$ as $\alpha \to 0$, where $x_\alpha = (A^* A +
  \alpha)^{-1} A^* y$, $\alpha > 0$, and $Q : Y \to \overline{R(A)}$
  is the orthogonal projection.
\end{lemma}

Since $W_\tau$ is bounded, the first Lemma implies that
(\ref{eqn:argminProb}) is solvable. To apply the second lemma to our
current setting, we must describe the range of $W_\tau$ and compute
$W_\tau^*1$. Toward that end, we recall that $\supp(W_\tau f) \subset
M(\tau)$ by finite speed of propagation. When $\tau$ is a step
function, Tataru's unique continuation \cite{Tataru1995} implies that
the inclusion
\begin{align}
  \label{density1}
  \{W_\tau f;\ f \in L^2(S_\tau)\} \subset L^2(M(\tau)),
\end{align}
is dense, see e.g. \cite[Th. 3.10]{Katchalov2001}. The result was
extended to the case of $\tau \in C(\overline{\Gamma})$ in
\cite{Oksanen2011}. Thus $\overline{R(W_\tau)} = L^2(M(\tau))$ for the
functions $\tau$ under consideration. To compute $W_\tau^* 1$, we note
an equality similar to (\ref{eq_Blago_identity}) that is satisfied for
$f \in L^2(S_\tau)$:
\begin{align}
  \label{eq_volume_identity}
  (u^f(T,\cdot), 1)_{L^2(M; \dV)} = (f,P_\tau b)_{L^2((0,T) \times \p
    M; dt \otimes dS_g)}.
\end{align}
Here, $P_\tau$ is defined by (\ref{eqn:pTau}), and $b(t, x) := T - t$
for $(t,x) \in (0,T) \times \Gamma$. Thus, $W_\tau^* 1 = P_{\tau}b$.

Applying Lemmas \ref{lemma:TikhExist} and \ref{lemma:TikhConv} to the
observations above, we see that for each $\alpha > 0$, equation
(\ref{eqn:argminProb}) has a unique solution $f_\alpha$, given by:
\begin{align}
  \label{tikhonov}
  f_\alpha := (W_\tau^* W_\tau + \alpha)^{-1} W_\tau^* 1 = (K_\tau +
  \alpha)^{-1} P_\tau b,
\end{align}
thus $f_\alpha$ is obtained from the data. Moreover, the waves $W_\tau
f_\alpha$ satisfy $W_\tau f_\alpha \to Q_\tau 1$ in $L^2(M)$ as
$\alpha$ tends to zero, where $Q_\tau$ is the projection of $L^2(M)$
onto the subspace $\overline{R(W_\tau)} = L^2(M(\tau))$.  Note that
$Q_\tau 1 = 1_{M(\tau)}$. Using this fact and applying
(\ref{eq_volume_identity}) to $f_\alpha$ we conclude,
\begin{align}
  \label{eq_reconstruction_vol}
  m(\tau) = \lim_{\alpha \to 0+} (f_\alpha , P_\tau b)_{L^2((0,T)
    \times \p M; dt \otimes dS_g)}.
\end{align}
Thus we can compute $m(\tau)$ from operations performed on the data
$\Lambda^{2T}_{\Gamma}$.

\section{Constructing distances}
\label{sec:ConstructingDistances}

In this section, we present our proof of Theorem
\ref{theorem:VolumesDetermineDistances}. We accomplish this through a
sequence of lemmas that are designed to illuminate the steps required
to turn the theorem into an algorithm. In addition, we provide an
alternative technique to determine distances, which we use in our
computational implementation.

\subsection{Constsructive proof of Theorem \ref{theorem:VolumesDetermineDistances}}

The following lemma provides a bound on the distance between a point
and a wave-cap,
\begin{lemma}
  \label{lemma:VolumeComparison}
  Let $y \in \Gamma$, $s \in (0, \sigma_{\Gamma}(y))$, and $h \in
  (0,\sigma_\Gamma(y) - s)$.  Let $z \in \Gamma$ and $r > 0$. Then
  $d(z,\wavecap_\Gamma(y,s,h)) < s+r$ if and only if
  \begin{align}
    \label{eqn:VolumeComparison}
    \meas(s 1_\Gamma + r 1_z + h 1_y) - \meas(s 1_\Gamma + r 1_z) < \meas(s 1_\Gamma + h
    1_y) - \meas(s 1_\Gamma).
  \end{align}
  We note that (\ref{eqn:VolumeComparison}) tests whether there is an
  overlap between the sets $\wavecap_\Gamma(y,s,h)$ and
  $\wavecap_\Gamma(y,s,r)$, see Figure \ref{fig:WavecapIntersections}.
\end{lemma}
\begin{proof}
As $h > 0$ and $h < \sigma_{\Gamma}(y) - s$, we see that
$\wavecap_\Gamma(y,s,h)$ contains a non-empty open set. In particular,
it has strictly positive measure.  Moreover, if
$d(z,\wavecap_\Gamma(y,s,h)) < s+r$ then the intersection of
$\wavecap_\Gamma(y,s,h)$ and $M(z,s+r)$ contains a non-empty open set
and has strictly positive measure.

Notice that $\meas(s 1_\Gamma + h 1_y)$ is the measure of $M(y,s+h)
\cup M(\Gamma,s)$ and that $\meas(s 1_\Gamma + h 1_y) - \meas(s
1_\Gamma)$ is the measure of $\wavecap_\Gamma(y,s,h)$.  Indeed,
\begin{align*}
  \Vol(M(y,s+h) \cup M(\Gamma,s)) &= \Vol((M(y,s+h) \cup M(\Gamma,s)) \setminus M(\Gamma,s))\\  
  &\quad+\Vol(M(\Gamma,s))  \\
  &= \Vol(\wavecap_\Gamma(y,s,h)) + \Vol(M(\Gamma,s)).
\end{align*}

Analogously, $\meas(s 1_\Gamma + r 1_z + h 1_y) - \meas(s 1_\Gamma + r
1_z)$ is the measure of
\begin{equation*}
  M(y,s+h) \setminus (M(\Gamma,s) \cup M(z,s+r)) = \wavecap_\Gamma(y,s,h) \setminus M(z,s+r).
\end{equation*}
If $d(z,\wavecap_\Gamma(y,s,h)) < s+r$ then the intersection of
$\wavecap_\Gamma(y,s,h)$ and $M(z,s+r)$ has strictly positive measure,
whence (\ref{eqn:VolumeComparison}) holds.

On the other hand, if $d(z,\wavecap_\Gamma(y,s,h)) \ge s+r$ then
$\wavecap_\Gamma(y,s,h) \cap M(z,s+r)$ is contained in the topological
boundary of $M(z,s+r)$ which is of zero measure
\cite{Oksanen2011}. Thus
\begin{align*}
  \meas(s 1_\Gamma + r 1_z + h 1_y) - \meas(s 1_\Gamma + r 1_z) =
  \meas(s 1_\Gamma + h 1_y) - \meas(s 1_\Gamma),
\end{align*}
and (\ref{eqn:VolumeComparison}) does not hold. \qquad
\end{proof}

\begin{figure}[h!] 
  \centering
  \begin{floatrow}
    \ffigbox[1.5in]
    {
      \subfloat[]{\includegraphics[height=1.4in]{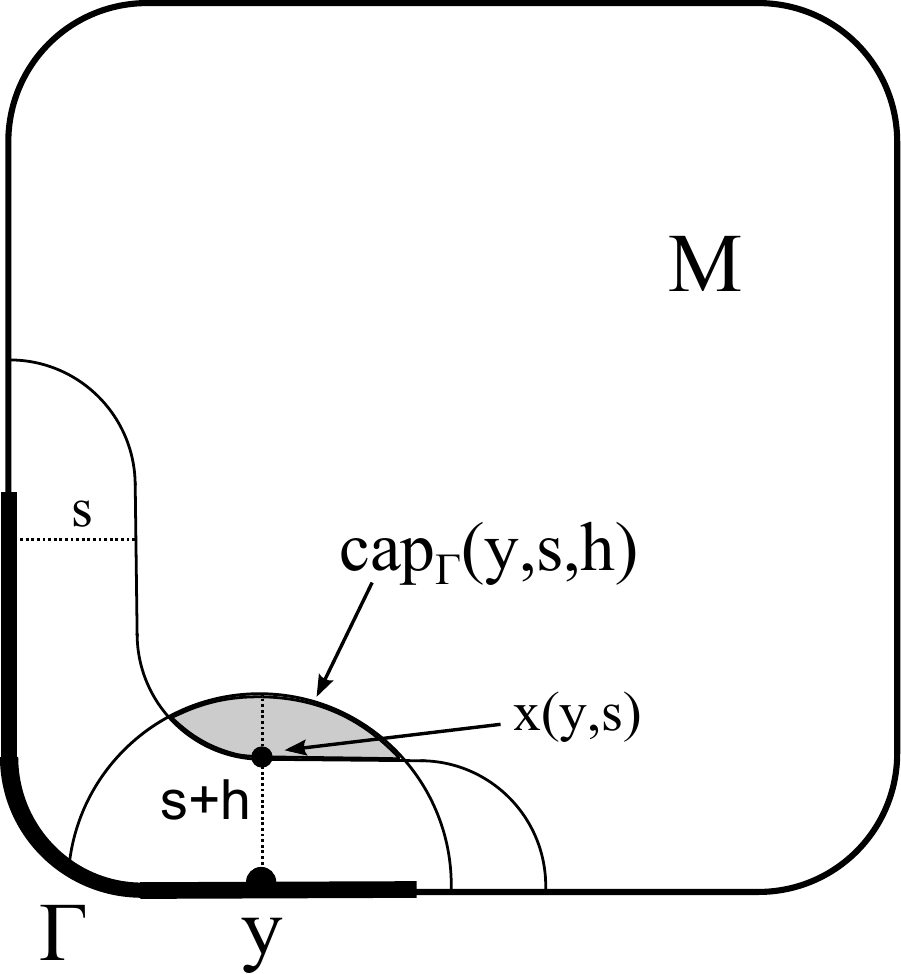}}
    }
    {
      \caption{A wave cap}
      \label{fig:WaveCap}
    }

    \ffigbox[3.6in]
    {
      \subfloat[$s+r > d(z,\wavecap_\Gamma(y,s,h))$]{
        \label{subflt:greaterthn}
        \includegraphics[height=1.4in]{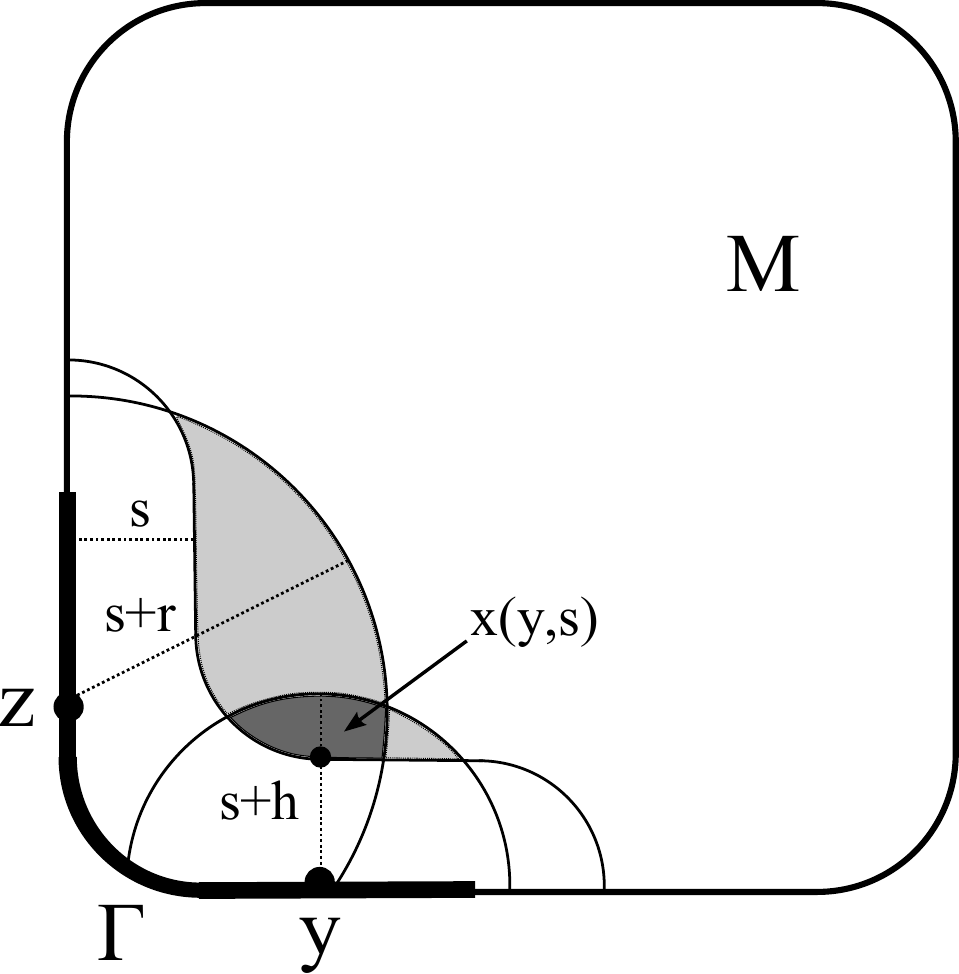}
      }
      \qquad
      \subfloat[$s+r \leq d(z,\wavecap_\Gamma(y,s,h))$]{
        \label{subflt:lessthn}
        \includegraphics[height=1.4in]{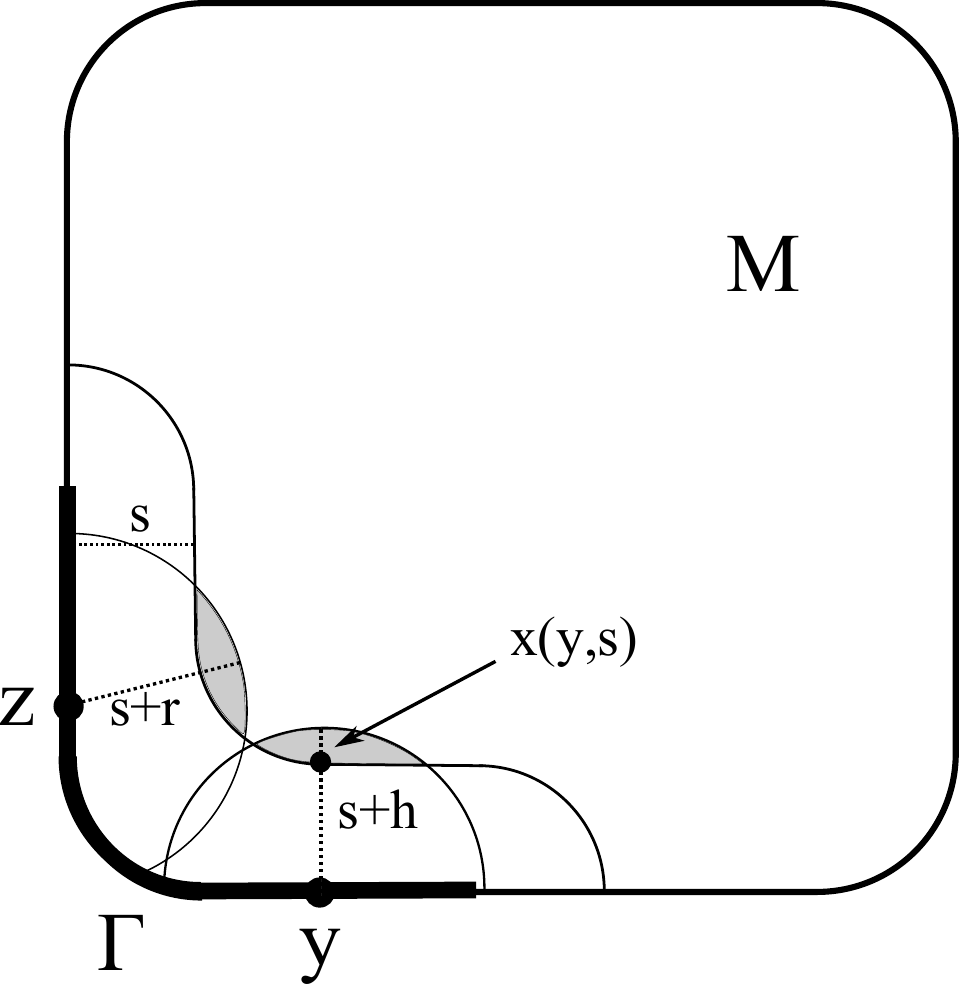}
      }
    }
    {
      \caption{ The light gray regions indicate the wave caps used in
        Lemma \ref{lemma:VolumeComparison} and the dark gray region
        indicates the overlap between the caps.  }
      \label{fig:WavecapIntersections}
    }
   \end{floatrow}
\end{figure}

The next lemma demonstrates that when $s < \sigma_\Gamma(y)$, the
wave caps $\wavecap_\Gamma(y,s,h)$ tend, in a set-theoretic sense,
towards $x(y,s)$. 

\begin{lemma}
  \label{lemma:wavecap_intersections}
  Let $y \in \Gamma$ and $s \in (0, \sigma_{\Gamma}(y))$.  Then, 
  \begin{align}
    \label{eq:wavecap_intersections}
    \bigcap_{h > 0} \wavecap_\Gamma(y,s,h) = \{x(y,s)\}.
  \end{align}
\end{lemma}

\begin{proof}
Let $y,s$ as above, and let $I(y,s)$ denote the left hand side of
(\ref{eq:wavecap_intersections}). Let $w$ be any point belonging to
$I(y,s)$. Then $w \in \wavecap_\Gamma(y,s,h)$ for all $h > 0$, so $s
\leq d(\Gamma,w)$ and $d(y,w) < s + h$ for all $h >0$, thus $d(y,w)
\leq s$. Since $y \in \Gamma$, we conclude $s = d(\Gamma,w) =
d(y,w)$. On the other hand, if $w$ is a point in $M$ satisfying $s =
d(\Gamma,w) = d(y,w)$, then $w \in \wavecap_\Gamma(y,s,h)$ for any $h
>0$, hence $w \in I(y,s)$. We conclude,
\begin{equation}
  I(y,s) = \{ w \in M : d(y,w) = d(\Gamma,w) = s\}.
\end{equation}
Because $s < \sigma_\Gamma(y)$, we have $d(x(y,s),y) =
d(x(y,s),\Gamma) = s$, so $x(y,s) \in I(y,s)$. It remains to show that
no other points belong to $I(y,s)$.

Let $w$ belong to $I(y,s)$, we will show that $w = x(y,s)$. If we knew
for certain that $w$ belonged to the image of the semi-geodesic
coordinates, then this would be immediate from the definition of
these coordinates. On the other hand, if we did not require $\Gamma$
to be open, then simple examples show that for points $y$ in the
topological boundary of $\Gamma$ it is possible that $I(y,s)$ has many
points. We demonstrate that when $\Gamma$ is open this cannot
happen.

Since $M$ is a compact connected metric space with distance arising
from a length function, the Hopf-Rinow theorem for length spaces
applies and we conclude that there is a minimizing path $\beta : [0,l]
\rightarrow M$ from $y$ to $w$. By \cite{Alexander1981}, $\beta$ is
$C^1$ and we may assume that it is unit speed parameterized. Hence $l
= s$. As $\beta$ is minimizing from both $y$ and $\Gamma$ to $w$, we
see that $\dot{\beta}(0) = \nu$. Thus $\beta$ coincides with $x(y,t)$
for $t \leq \min(s, \tau_M(y,\nu))$. But $s < \sigma_\Gamma(y)$, hence
$s < \tau_M(y,\nu)$. Thus we see that $w = \beta(s) = x(y,s)$. \qquad
\end{proof}

We use the preceding lemma to show that, when $h$ is small, the
distance between a point $z \in \Gamma$ and the wave cap
$\wavecap_\Gamma(y,s,h)$ surrounding $x(y,s)$ yields an approximation
to $d(z,x(y,s))$. We depict this approximation in Figure
\ref{figure:capDistanceApprox}.

\begin{figure}[!ht]
  \includegraphics[height =1.4in]{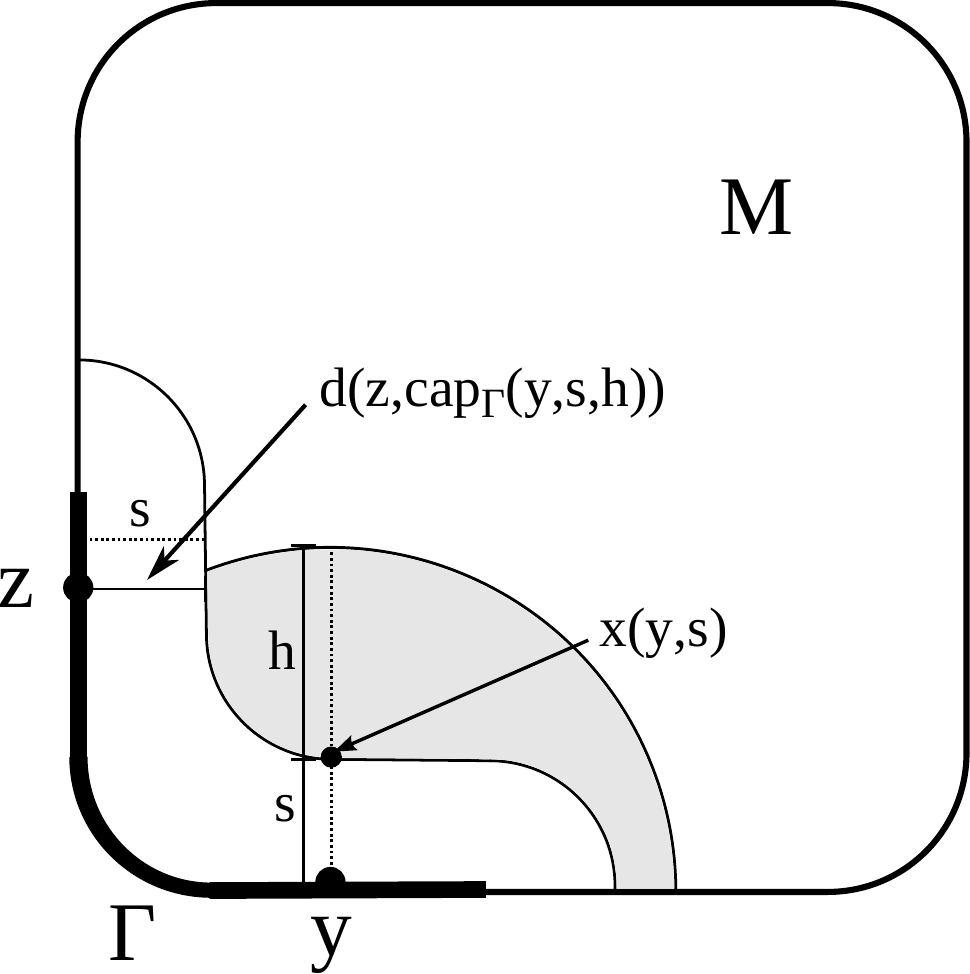}
  \hspace*{.2in}
  \includegraphics[height =1.4in]{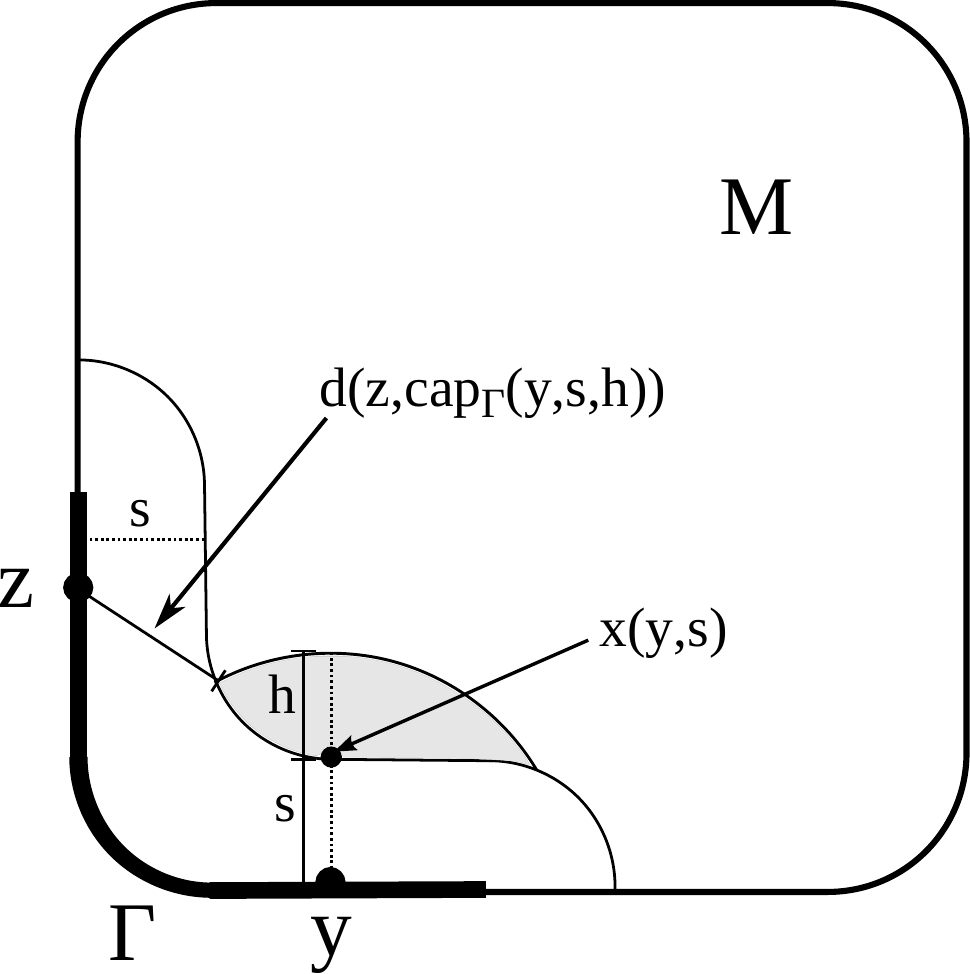}
  \hspace*{.2in}
  \includegraphics[height =1.4in]{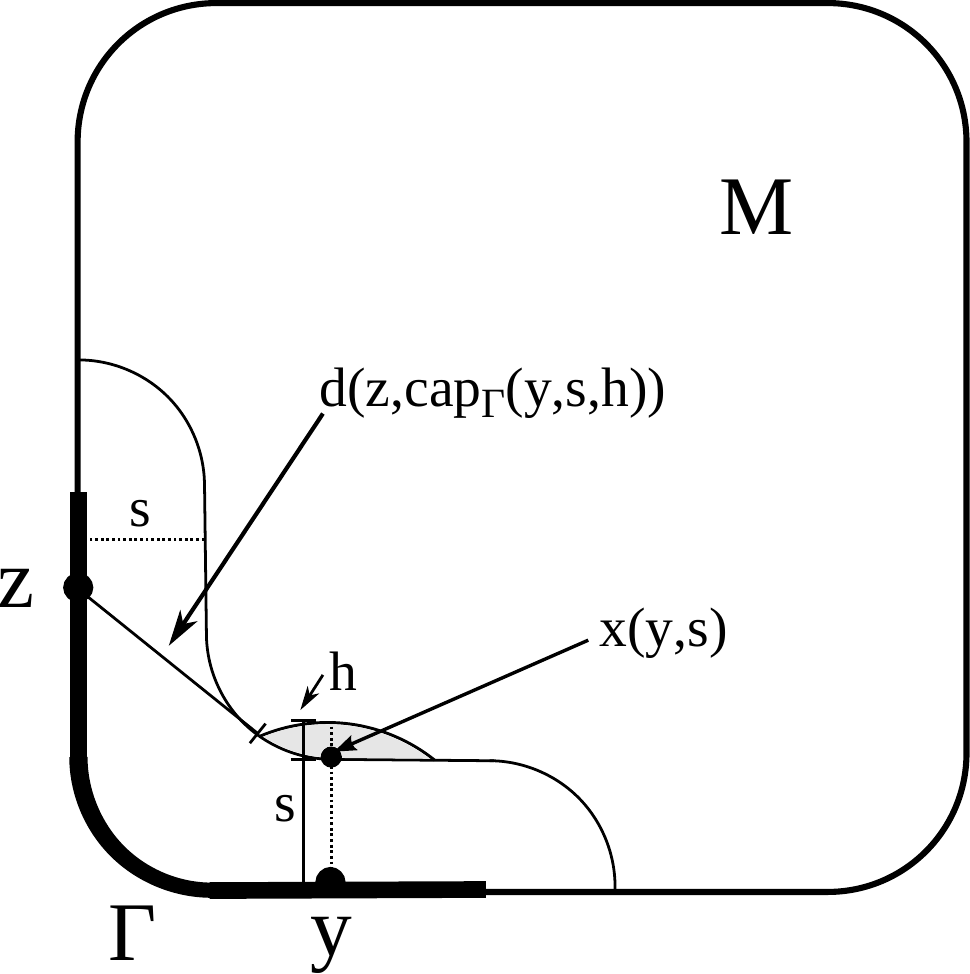}
  \caption{Cartoon demonstrating that $d(z,\wavecap_\Gamma(y,s,h))
    \rightarrow d(z,x(y,s))$ as $h \rightarrow 0$.}
  \label{figure:capDistanceApprox}
\end{figure}

\begin{lemma}
  \label{lemma:CapDistanceApprox}
  For $y,z \in \Gamma$, and $s < \sigma_\Gamma(y)$,
  $d(z,\wavecap_\Gamma(y,s,h)) \rightarrow d(z,x(y,s))$ as $h \rightarrow 0$.
\end{lemma}
\begin{proof}
  Let $\{h_j\} \subset \R_+$ be a sequence for which $h_j \downarrow
  0$. Then each $\wavecap_\Gamma(y,s,h_j)$ is compact, so there exists
  $w_j \in \wavecap_\Gamma(y,s,h_j)$ such that $d(z,w_j) =
  d(z,\wavecap_\Gamma(y,s,h_j))$. Because $M$ is a compact manifold,
  the sequence $\{w_j\}$ has a convergent subsequence $\{w_{j_k}\}$
  converging to a point $w$. Since the wave caps
  $\wavecap_\Gamma(y,s,h)$ nest, the tail of $\{w_{j_k}\}$ belongs to
  the closed set $\wavecap_\Gamma(y,s,h_{j_k})$ for each $j_k$, hence
  $w \in \wavecap_\Gamma(y,s,h)$ for each $h > 0$. By the previous
  lemma, we conclude $w = x(y,s)$.
  
  Together, continuity of the distance function and the particular
  choice of the $w_{j_k}$ imply that, $ d(z,\wavecap_\Gamma(y,s,h_{j_k}))
  = d(z,w_{j_k}) \rightarrow d(z,x(y,s)).$  Since the wave caps nest,
  the sequence $\{d(z,w_j)\}$ is monotone non-decreasing, and since it
  is bounded above it has a limit. In particular, any subsequential
  limit coincides with the limit. Thus we conclude that
  $d(z,\wavecap_\Gamma(y,s,h_j)) \rightarrow d(z,x(y,s))$ as $j
  \rightarrow \infty$, and in turn, $d(z,\wavecap_\Gamma(y,s,h))
  \rightarrow d(z,x(y,s))$ as $h \rightarrow 0$. \qquad
\end{proof}

The volumes appearing in Lemma \ref{lemma:VolumeComparison} cannot be
computed directly with the $\tau$'s appearing in the regularized
volume determination. That is, the lemma requires us to compute
volumes such as $m(s1_\Gamma + r1_z + h1_y)$, but the function $\tau =
s1_\Gamma + r1_z + h1_y$ is equivalent to $s1_{\Gamma}$ in
$L^2((0,T)\times \Gamma)$. As a result, the set $L^2(S_\tau)$ will
produce waves that fill $L^2(M(s1_{\Gamma}))$ as opposed to the
desired set $L^2(M(\tau))$. The problem is that the spikes $h1_y$ and
$r1_z$ have supports with $\dS$ measure zero. The remedy is to replace
the spikes by functions that produce the same domains of influence but
have better supports. To accomplish this, for $y \in \Gamma$ and $R
\in [0,\infty)$, we define $\tau_y^R$ on $\Gamma$ by:
\begin{equation}
  \tau_y^R(z) := R - d(z,y) \text{\quad for $z \in \Gamma$}.
\end{equation} 
Note that $\tau_y^R$ is continuous.  We recall that under Assumption
\ref{assuption:distances} the distances $d(y,z)$ for $y,z \in \Gamma$
with $d(y,z) < T$ are known (or, alternatively, that they have been
computed in some other fashion from $\Lambda_\Gamma^{2T}$). Thus under
our assumptions the functions $\tau_y^R$ are known.

\begin{lemma}
\label{lemma:SpikeReplacements}
Let $y,z \in \Gamma$, $s,r,h > 0$. We will use the notation $f \vee g$
to denote the function obtained by taking the pointwise maximum of $f$
and $g$. Then, we have the following equalities,
\begin{align} 
    \label{eq:1stEq} M(\tau_y^r) &= M(r1_y),\\ 
    \label{eq:2ndEq} M(\tau_y^{s+h} \vee \tau_z^{s+r} \vee s) &= M(h1_y + r1_z + s1_\Gamma),\\
    \label{eq:3rdEq} M( \tau_z^{s+r} \vee s) &= M(s1_\Gamma + r1_z).
\end{align}
\end{lemma}
\begin{proof}
  Let $x \in M(r1_y)$, then $d(y,x) < r$. Since $\tau_y^r(y) = r$, we
  have that $d(y,x) < \tau_y^r(x)$, hence $x \in M(\tau_y^r)$.  Now
  let $x \in M(\tau_y^r)$. Then there is a point $z \in \Gamma$ for
  which $d(x,z) < \tau_y^r(z)$. Applying the definition of $\tau_y^r$,
  we find $r > d(x,z) + d_\Gamma(y,z) \geq d(x,z) + d(y,z) \geq
  d(x,y)$. Hence $x \in M(r1_y)$. We conclude that $M(\tau_y^r) =
  M(r1_y)$.

  We demonstrate equality (\ref{eq:2ndEq}) and note that
  (\ref{eq:3rdEq}) is proved in an analogous fashion. Let $\tau =
  \tau_y^{s+h} \vee \tau_z^{s+r} \vee s$. Then $x \in M(\tau)$ just in
  case $d(x,p) < \tau(p)$ for some $p \in \Gamma$, which happens if
  and only if $d(x,p)$ is less than $\tau_y^{s+h}(p),\tau_z^{r+h}(p)$,
  or $s$. The preceding paragraph implies that this happens just in
  case $x$ belongs to $M((s+h)1_y), M((s+r)1_z),$ or $M(s1_\Gamma)$,
  which happens if and only if $x \in M(s1_y + r1_z + s1_\Gamma)$. \qquad
\end{proof}

We are finally in a position to prove Theorem
\ref{theorem:VolumesDetermineDistances}.

{\em Proof}. [Of Theorem \ref{theorem:VolumesDetermineDistances}]
First, let $r$ and $h$ be positive numbers satisfying $s + r <T$ and
$s + h < T$. Define functions $\tau_1 = s$, $\tau_2 = \tau_y^{s+h}
\vee s$, $\tau_3 = \tau_z^{s+r} \vee s$, and $\tau_4 = \tau_y^{s+h}
\vee \tau_z^{s+r} \vee s$. Using the regularized volume determination
from equation (\ref{eq_reconstruction_vol}), we compute the volumes
$m(\tau_i)$ for $i = 1,\ldots,4$.  Then, Lemma
\ref{lemma:SpikeReplacements} implies that $m(\tau_1) = m(s1_\Gamma)$,
$m(\tau_2) = m(s1_\Gamma + h1_y)$, $m(\tau_3) = m(s1_\Gamma + r1_z)$,
and $m(\tau_4) = m(s1_\Gamma + h1_y + r1_z)$, thus we have determined
the volumes appearing in (\ref{eqn:VolumeComparison}). By Lemma
\ref{lemma:VolumeComparison} we can compute
$d(z,\wavecap_\Gamma(y,s,h))$ by
\begin{equation}
d(z,\wavecap_\Gamma(y,s,h)) = s + \inf \{r : 0 \leq r < T - s, \text{ and
  (\ref{eqn:VolumeComparison}) holds\}}.
\end{equation}  
Finally, by Lemma \ref{lemma:CapDistanceApprox}, we can compute
$d(z,x(y,s))$ by
\begin{equation}
  d(z,x(y,s)) = \lim_{h\rightarrow 0} d(z,\wavecap_\Gamma(y,s,h)). \qquad\endproof
\end{equation}

\subsection{Alternative distance estimation method}

The method to estimate distances derived from Theorem
\ref{theorem:VolumesDetermineDistances} uses the fact that, under the
hypotheses of the theorem, the distance between a point $z \in \Gamma$
and the wave cap $\wavecap_\Gamma(y,s,h)$ serves as an approximation
to $d(z,x(y,s))$, and that this approximation improves as $h
\rightarrow 0$. However, in the case where $g$ is the Euclidean
metric, $d(z,\wavecap_\Gamma(y,s,h))$ converges to $d(z,x(y,s))$ with
the rate $\mathcal{O}(h^{1/2})$. Thus the convergence is typically
slow. In this section, we provide another technique to estimate the
distance to points which we find, for a given nonzero $h$, tends to
provide better distance estimates.

The idea of this alternative distance estimation method is to once
again check for overlap between the sets $\wavecap_\Gamma(y,s,h)$ and
$\wavecap_\Gamma(z,s,r)$, but instead of seeking the minimum $r$ for
which these wave caps overlap, we seek $r$ for which
$\Vol(\wavecap_\Gamma(y,s,h) \cap \wavecap_\Gamma(z,s,r))$ is half of
$\Vol(\wavecap_\Gamma(y,s,h))$.

Before proving that our alternative distance estimation procedure is
valid, we provide a lemma that shows that the diameter of a wave cap
vanishes as the height of the cap goes to zero.

\begin{lemma}
  Let $y,z \in \Gamma$, $s \in (0,\sigma_{\Gamma}(y))$. Then,
  \begin{equation}
    \lim_{h \rightarrow 0} \diam(\wavecap(y,s,h)) = 0.
  \end{equation}
\end{lemma}
\begin{proof}
  Suppose the claim were false. Then there exists a sequence of
  positive real numbers $h_i \downarrow 0$ and points $p_i \in
  \wavecap_\Gamma(y,s,h_i)$ such that $d(x(y,s),p_i) \not \rightarrow
  0$.  Since $M$ is compact, the sequence $\{p_i\}$ has a convergent
  subsequence. Relabeling this subsequence by $p_i$ we have that there
  exists $p \in M$ such that $p_i \rightarrow p$. But this implies
  that $d(p,x(y,s)) \neq 0$, hence $p \neq x(y,s)$. On the other hand,
  since $p_i \in \wavecap_\Gamma(y,s,h_i)$ we must have that $p \in
  \bigcap_{h>0} \wavecap_{\Gamma}(y,s,h)$, but this gives a
  contradiction, since by Lemma \ref{lemma:wavecap_intersections} this
  implies that $p = x(y,s)$. \qquad
\end{proof}

We now present our alternative distance estimation method.

\begin{lemma}
  \label{lemma:alternateDistanceProxy}
  Let $y,z \in \Gamma$, $s \in (0,\sigma_{\Gamma}(y))$, and $0 < h <
  \sigma_{\Gamma}(y) - s$. Let $r_h$ be the solution to,
  \begin{equation}
    \label{eqn:alternateDistanceProxy}
    \Vol(\wavecap_\Gamma(y,s,h) \cap \wavecap_\Gamma(z,s,r_h)) = \frac{1}{2}
    \Vol(\wavecap_{\Gamma}(y,s,h)).
  \end{equation}
  Then, for $d_h := s + r_h$, we have that $d_h
  \rightarrow d(z,x(y,s))$ as $h \rightarrow 0$.
\end{lemma}
\begin{proof}
  First, we recall that for $s$ and $h$ as above,
  $\wavecap_\Gamma(y,s,h)$ will contain a non-empty open set, hence
  the right-hand side of (\ref{eqn:alternateDistanceProxy}) will be
  nonzero. Thus, from the definition of $r_h$, we conclude that
  $\wavecap_\Gamma(y,s,h) \cap \wavecap_\Gamma(z,s,r_h)$ is a
  non-empty and proper subset of $\wavecap_\Gamma(y,s,h)$. Using the
  definition of $\wavecap_\Gamma(z,s,r_h)$ and $r_h$ we conclude that
  $s + r_h \geq d(z,\wavecap_\Gamma(y,s,h))$. On the other hand, since
  the intersection between the wave caps is a proper subset of
  $\wavecap_\Gamma(y,s,h)$ we see that there exists $p \in
  \wavecap_\Gamma(y,s,h) \setminus \wavecap_\Gamma(z,s,r_h)$. In
  particular, this implies that $s + r_h \leq d(z,p) \leq
  \dist(z,\wavecap_\Gamma(y,s,h)) +
  \diam(\wavecap_\Gamma(y,s,h))$. Hence,
  \begin{equation*}
    d(z,\wavecap_\Gamma(y,s,h)) \leq d_h \leq
    d(z,\wavecap_\Gamma(y,s,h)) + \diam(\wavecap_\Gamma(y,s,h)).
  \end{equation*}
  Since $d(z,\wavecap_\Gamma(y,s,h)) \rightarrow d(z,x(y,s))$ and
  $\diam(\wavecap_\Gamma(y,s,h)) \rightarrow 0$ as $h \rightarrow 0$,
  we conclude that $d_h \rightarrow d(z,x(y,s))$ as $h
  \rightarrow 0$. \qquad
\end{proof}

We summarize the steps of the proof in an algorithmic form in Algorithm
\ref{algo:distanceDetermination}.

\label{sec:ContainsAlgorithm}
\begin{algorithm}
  \kwLet{$y,z \in \Gamma$ and $s > 0$ with $r_{x(y,s)}(z) < T.$}

  \kwLet{$h_0 > 0$ small enough that $s + h_0 < \min\{\sigma_\Gamma(y), T\}$.}

  \For{$0 < h < h_0$}{
    \For{$0 < r < T - s$}{
      \kwLet{
          $\tau_1 = s1_\Gamma$, $\tau_2 = \tau_{y}^{s+h}$, $\tau_3 =
          \tau_{z}^{s+r}$, $\tau_4 = \tau_1 \vee \tau_2 \vee \tau_3$
      }
      \For{$\alpha > 0$}{
        \For{$i = 1,\ldots,4$}{
          \kwLet{ $f_{\alpha,i}$ solve:
            \begin{equation*}          
              (K_{\tau_i} + \alpha) P_{\tau_i} f = P_{\tau_i} b
            \end{equation*}
          }
        }
      }
      \For{$i = 1,\ldots,4$}{
        \kwCompute{
        \begin{equation*}
          m(\tau_i) = \lim_{\alpha \rightarrow 0} (f_{\alpha,i}, b)_{L^2(S_\tau; \,dt \otimes dS_{g,\mu})}
        \end{equation*}
        }
      }
      \kwCompute{
        \begin{equation*}
            \begin{array}{rcl}
                m_{\text{target cap}}(h) &:=& m(\tau_2) - m(\tau_1) \\
                m_{\text{overlap}}(h,r) &:=& m(\tau_4) - m(\tau_3) - m(\tau_2) + m(\tau_1)\\
            \end{array}
        \end{equation*}        
      }
    }
    \kwCompute{$r_h$ by either:      
      \begin{quote}
        \textbf{method 1}: let $r_h$ satisfy:
        \begin{equation*}
          r_h = \inf \{r > 0: m_{\text{overlap}}(h,r) > 0\}.
        \end{equation*}
          
        \textbf{method 2}: let $r_h$ be the solution to:
        \begin{equation*}
	  m_{\text{overlap}}(h,r) = \frac{1}{2} m_{\text{target cap}}(h).
        \end{equation*}
      \end{quote}
    }
  }
  \kwCompute{ $r_{x(y,s)}(z)$ by:
    \begin{equation*}
      r_{x(y,s)}(z) = s + \lim_{h \rightarrow 0} r_h.
    \end{equation*}
  }
  \caption{
    \label{algo:distanceDetermination}
    Continuum level description of distance determination algorithm.
  }
\end{algorithm}

\section{Computational experiment}
\label{sec:Numerics}

In this section we present a numerical example that demonstrates the
distance determination procedure that we have described in the
previous sections. For computational simplicity, we demonstrate our
procedure in the Euclidean setting. However, we stress that our method
can be applied in the general Riemannian setting.

\subsection{Numerical method for the direct problem}

For our numerical example, we take the manifold $M$ to be the
$2$-dimensional Euclidean lower half-space equipped with the canonical
metric and a unit weight function, $\mu \equiv 1$. Under these
particular choices, the weighted Laplace-Beltrami operator reduces to
the Euclidean $2$-dimensional Laplacian. Hence, for our example, the
Riemannian wave equation (\ref{eqn:ForwardProb}) simplifies to the
standard $2+1$-dimensional wave equation with constant sound-speed,
$c\equiv 1$.  In order to simulate the situation of partial, local
illumination, for our source/receiver set, $\Gamma$, we take $\Gamma =
[-L,L] \times \{0\} \subset \p M$ with $L = 2.232$. We simulate waves
propagating for $2T$ time units, where $T = 1.249$.


For sources, we use a basis of Gaussian pulses with the form
\begin{equation*}
  \varphi_{i,j}(t,x) = C \exp\left(-a_t (t-t_i)^2 -a_x
  (x-x_j)^2\right),
\end{equation*}
with parameters $a_t = a_x = 4 \cdot 10^3$, and where we have selected
the constant $C$ to normalize the $\varphi_{i,j}$.  Sources are
applied on the regular grid:
\begin{equation*}
  \left\{ (t_{s,i}, (x_{s,j},0)) :
  \begin{array}{ll}
     t_{s,i} = t_{s,1} + (i-1) \Delta t_s & i=1,\ldots,N_{t,s}, \\
     x_{s,j} = x_{s,1} + (j-1)\Delta x_s & j = 1,\ldots,N_{x,s}
  \end{array}
  \right\},
\end{equation*}
where the source offset $\Delta x_s$ and time between source
applications $\Delta t_s$ are both selected as $\Delta x_s = \Delta
t_s = .0147$. At each of the $N_{x,s} = 309$ source positions we apply
$N_{t,s} = 78$ sources. For each basis function, we record the
Dirichlet trace data on the regular grid:
\begin{equation*}
  \left\{ (t_{r,l}, (x_{r,k},0)) :
  \begin{array}{ll}
      t_{r,l} = t_{r,1} + (l-1) \Delta t_r & l=1,\ldots,N_{t,r}, \\
      x_{r,k} = x_{r,1} + (k-1)\Delta x_r & k = 1,\ldots,N_{x,r}
  \end{array}
  \right\}.
\end{equation*}
The receiver offset $\Delta x_r$ has been taken to be half the source
offset, resulting in $N_{x,r} = 633$ receiver positions.  The time
between receiver measurements, $\Delta t_r$, is $1/10$ the source time
between source applications, resulting in $N_{t,r} = 1701$ receiver
measurements at each receiver position.


We discretize the Neumann-to-Dirichlet map by solving the forward
problem for each source $\varphi_{i,j}$ and recording its Dirichlet
trace at the receiver positions and times described above. That is, we
compute the following data,
\begin{equation}
  \label{eqn:DiscreteNtD}
  \left\{ \Lambda^{2T}\varphi_{i,j}(t_{r,l},x_{r,k}) =
  u^{\varphi_{i,j}}(t_{r,l},x_{r,k}) :
  \begin{array}{l}
    i=1,\ldots,N_{t,s}, ~j = 1,\ldots,N_{x,s},\\ l=1,\ldots,N_{t,r},
    ~k = 1,\ldots,N_{x,r}
  \end{array}
  \right\}.
\end{equation}

To perform the forward modelling, we use the continuous Galerkin
finite element method with piecewise linear Lagrange polynomial
elements and implicit Newmark time-stepping. In particular, we use the
FEniCS package \cite{LoggMardalEtAl2012a}.  We use a regular
triangular mesh, where the time step and mesh spacing are selected so
that $8$ points per wavelength (in directions parallel to the grid
axes) are used at the frequency $f_0$ where the spectrum of the
temporal portion of the source falls below $10^{-6}$ times its maximum
value.

\subsection{Solving the control problem}
\label{subsec:SolvingTheControlProb}

We discretize the connecting operator $K$ by approximating its action
as an operator on $\linearSpan\{\varphi_{i,j}\}$. That is, we use the
discrete Neumann-to-Dirichlet data, (\ref{eqn:DiscreteNtD}), to
discretize $K_\tau$ by formula (\ref{eq_connecting}), where $\tau
\equiv T$. To be specific, we first compute the Gram matrix $[G]_{ij}
= (\varphi_i,\varphi_j)$ and its inverse $[G^{-1}]$, and then compute
the matrices for the operators $J\Lambda_{\Gamma}^{2T}$,
$R\Lambda_{\Gamma}^{T}$ and $RJ$ acting on
$\linearSpan\{\varphi_{i,j}\}$ as follows:
\begin{align*}
  [J\Lambda_{\Gamma}^{2T}]_{ij} &= \sum_k [G^{-1}]_{ik} (\varphi_k, J \Lambda_{\Gamma}^{2T} \varphi_j),\\
  [R\Lambda_{\Gamma}^T]_{ij} &= \sum_k [G^{-1}]_{ik} (\varphi_k, R \Lambda_{\Gamma}^{T} \varphi_j),\\
  [RJ]_{ij} &= \sum_k [G^{-1}]_{ik} (\varphi_k, RJ \varphi_j).
\end{align*}
Using these operators we can compute the matrix for $K$,
\begin{equation}
  \label{eqn:discreteBlago}
  [K] = [J\Lambda_{\Gamma}^{2T}] - [R \Lambda_{\Gamma}^T] [RJ].
\end{equation}
In Table \ref{table:BlagoPictures} we demonstrate the effect of the
constituent matrices in (\ref{eqn:discreteBlago}) on a source $f$ and
how these combine to form $Kf$.

\begin{table}
  \centering
  \begin{tabular}{rc}
    (a)   &
    \raisebox{-.5\height}{\includegraphics[height = .85in]{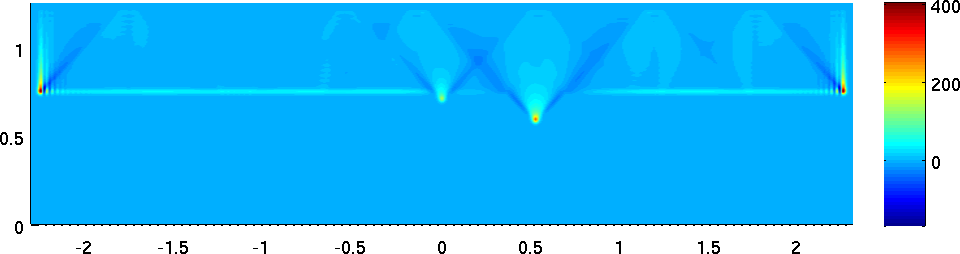}} \\
    (b) & 
    \raisebox{-.5\height}{\includegraphics[height = .85in]{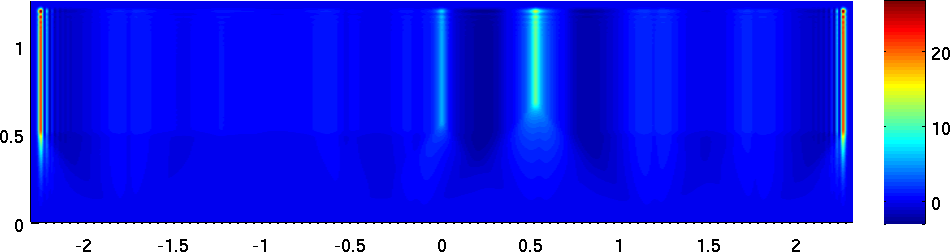}} \\
    (c) & 
    \raisebox{-.5\height}{\includegraphics[height = .85in]{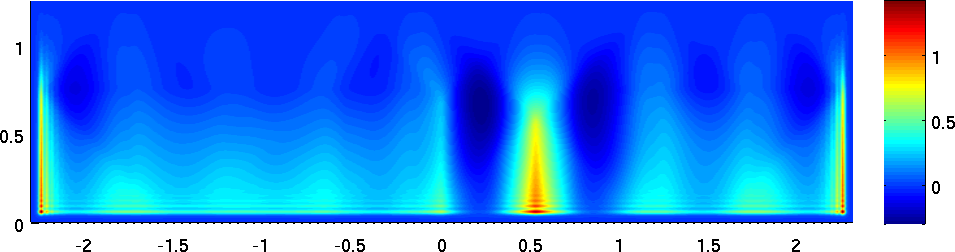}} \\
    (d) & 
    \raisebox{-.5\height}{\includegraphics[height = .85in]{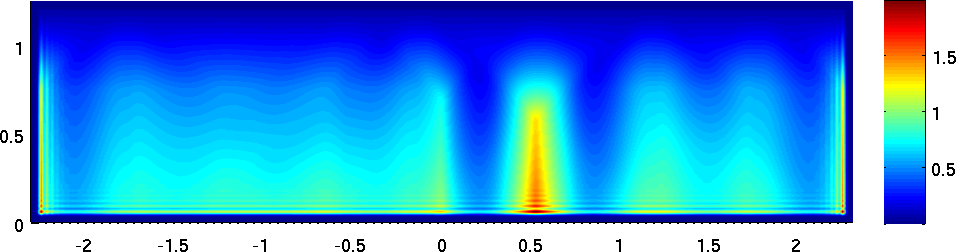}} \\
    (e) & 
    \raisebox{-.5\height}{\includegraphics[height = .85in]{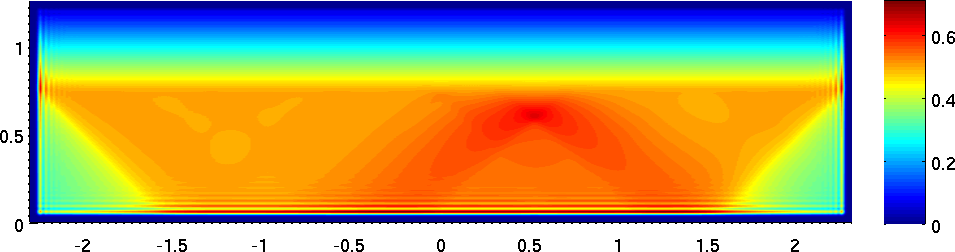}} \\
  \end{tabular}
  \caption{
    \label{table:BlagoPictures} 
    Demonstration of the action of the operators appearing in
    (\ref{eqn:discreteBlago}), (a) a source $f$, (b) $RJf$, (c)
    $R\Lambda_{\Gamma}^T RJf$, (d) $J\Lambda_{\Gamma}^{2T} f$, and (e)
    $Kf$. Note that all plots are on the function level, and each
    function depicted was obtained from coefficients that were
    computed as described in Section~\ref{subsec:SolvingTheControlProb}.}
\end{table}

For $\tau \in C(\Gamma)$, with $0 \leq \tau \leq T$, we obtain the
matrix $[K_\tau]$ discretizing the connecting operator $K_\tau$ by
masking the entries in $[K]$ that correspond to basis functions
$\varphi_{i,j}$ with centers $(t_{s,i},x_{s,j}) \not \in S_\tau$. We
note that, in practice, we find that this tends to provide a better
approximation to $K_\tau$ than computing the matrix $[P_\tau]$ and
computing the product $[P_\tau][K][P_\tau]$.

We consider the discretized control problem
\begin{equation}
  \label{eqn:discreteControlProb}
  ([K_\tau] + \alpha [P_\tau]) [f_\alpha] = [P_\tau] [b],
\end{equation} 
where we use the matrix $[P_\tau]$ to refer to the mask described
above, and use $\alpha = 10^{-5}$. Recall that $b$ is the function
given by $b(t,x) = T - t$, as defined beneath
(\ref{eq_volume_identity}). To solve (\ref{eqn:discreteControlProb})
for $[f_\alpha]$, we use restarted GMRES. In Figure
\ref{fig:controlAndWavefield} and Table \ref{tbl:SourcesAndWavefields}
we depict control solutions $f_\alpha = \sum_{i} [f_\alpha]_i
\varphi_i$ and their associated wavefields $u^{f_\alpha}(T,\cdot)$. A
volume estimate $\hat{m}(\tau)$ for $m(\tau)$ is obtained from
$[f_\alpha]$ by computing the discretized inner product $\hat{m}(\tau)
= [f_\alpha]^T [G] [b]$, which approximates $m(\tau)$ as in
(\ref{eq_reconstruction_vol}). For the remainder of this paper we will
continue to use the notation $\hat{m}(\tau)$ to indicate the
approximation to $m(\tau)$ computed in this fashion.

\begin{figure}[h]
    \includegraphics[height = 2.0in]{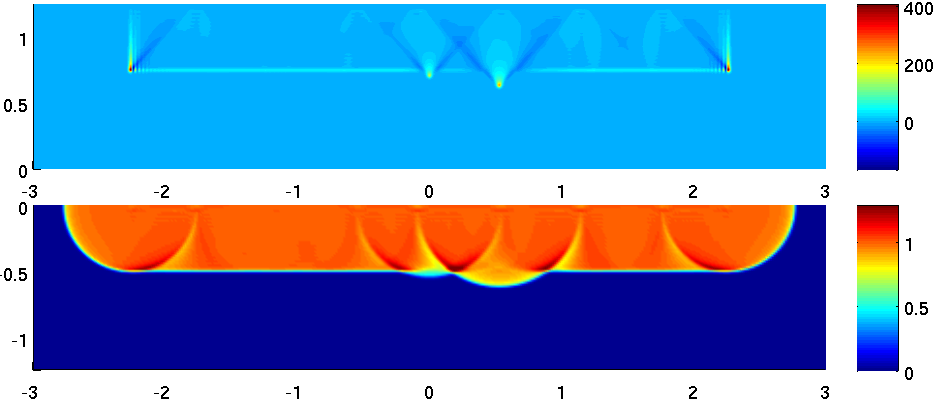}
    \caption{
      \label{fig:controlAndWavefield}
      Illustration of a source $f_\alpha$ and the corresponding
      wavefield $u^{f_\alpha}(T,\cdot)$ for which
      $u^{f_\alpha}(T,\cdot) \approx 1_{M(\tau)}$. Here, $\tau$
      corresponds to $\tau_{4,1}$ from Table
      \ref{tbl:SourcesAndWavefields}. Note that, to show both plots
      with the same horizontal axis, we have extended $f_\alpha$ to
      zero outside of $[0,T] \times \Gamma$.  }
\end{figure}

\subsection{Estimating distances}

\begin{table}
  \centering
  \begin{tabular}{rcc}
    \toprule
    $\tau$ & $f_\alpha$ & $u^{f_\alpha}(T,\cdot)$ \\
    \midrule
    $\tau_1$ &
    \raisebox{-.5\height}{\includegraphics[height = .69in]{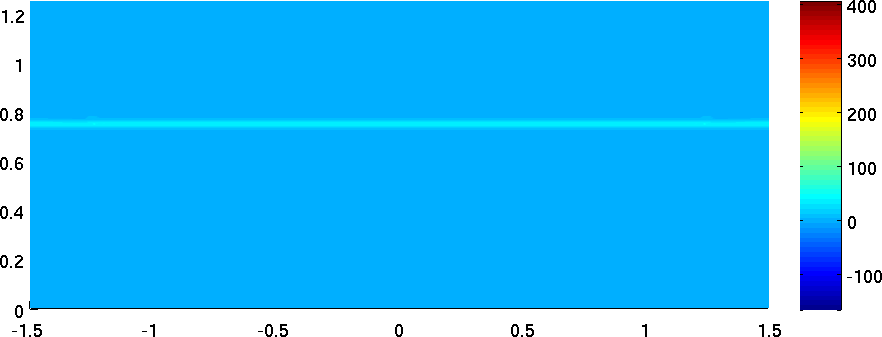}} &
    \raisebox{-.5\height}{\includegraphics[height = .69in]{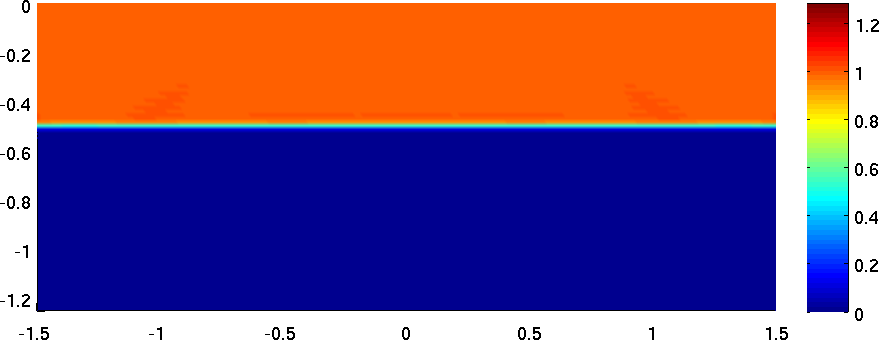}} \\
    $\tau_2$ &
    \raisebox{-.5\height}{\includegraphics[height = .69in]{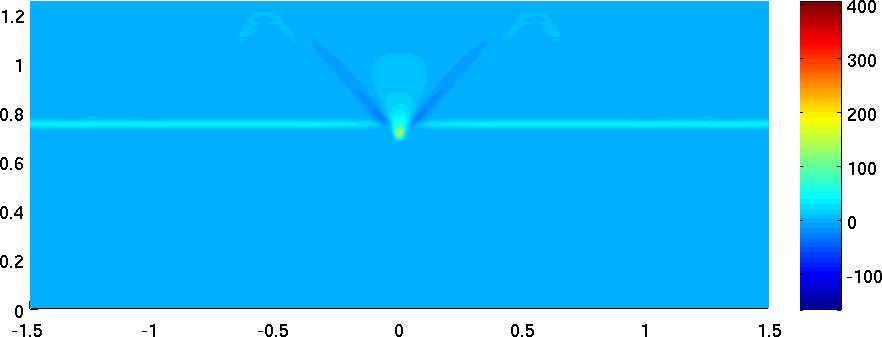}} &
    \raisebox{-.5\height}{\includegraphics[height = .69in]{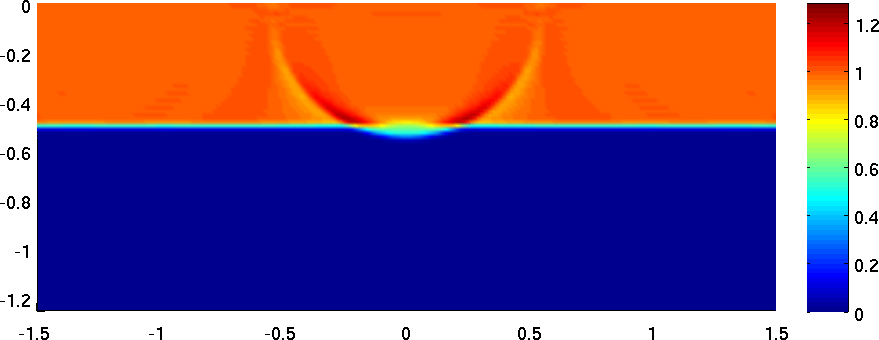}} \\
    \midrule
    $\tau_{3,1}$ &
    \raisebox{-.5\height}{\includegraphics[height = .69in]{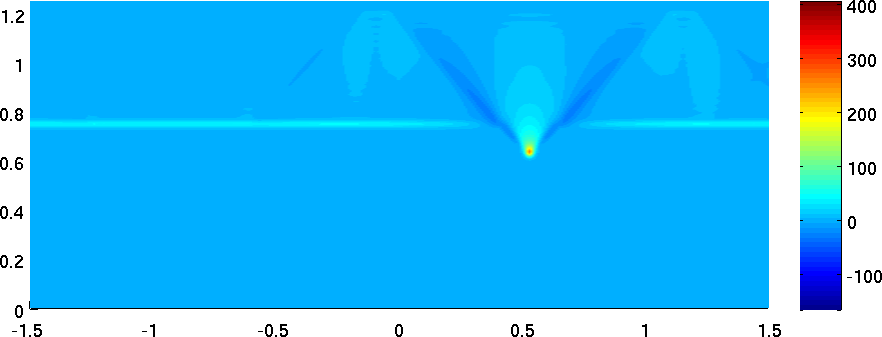}} &
    \raisebox{-.5\height}{\includegraphics[height = .69in]{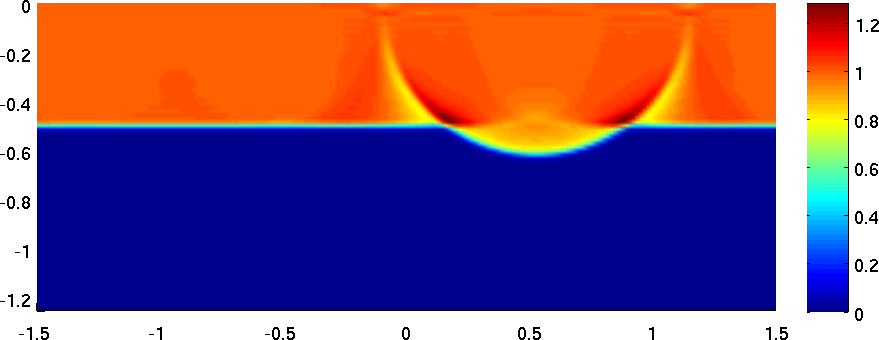}} \\
    $\tau_{4,1}$ &
    \raisebox{-.5\height}{\includegraphics[height = .69in]{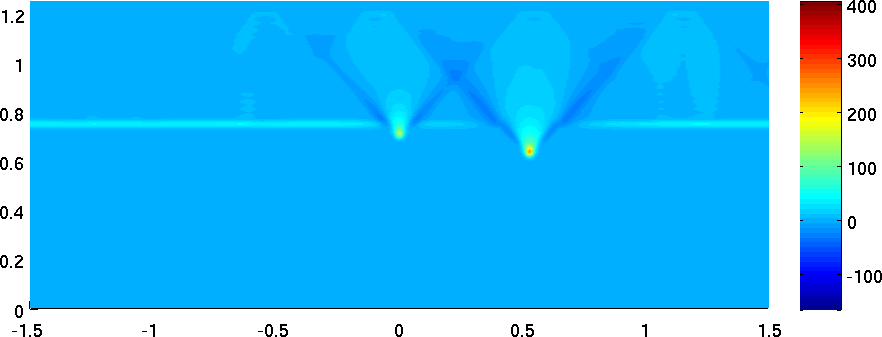}} &
    \raisebox{-.5\height}{\includegraphics[height = .69in]{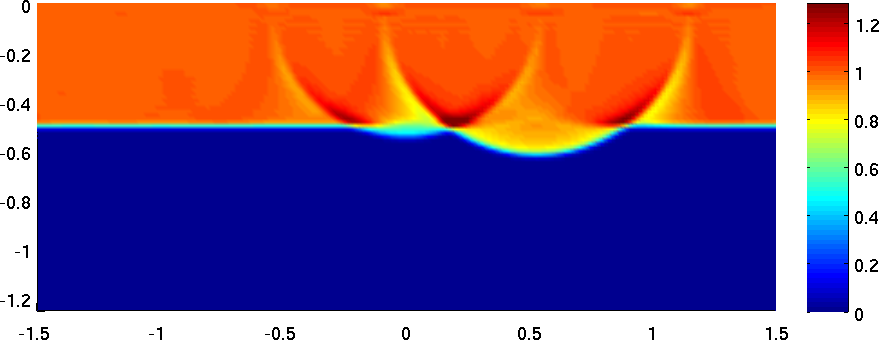}} \\
    \midrule
    $\tau_{3,2}$ &
    \raisebox{-.5\height}{\includegraphics[height = .69in]{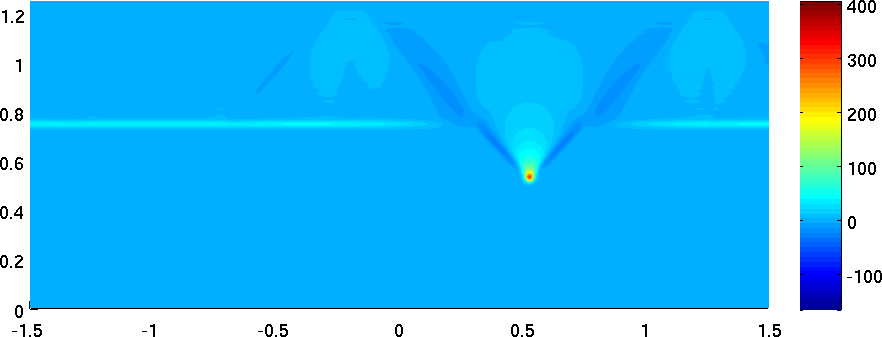}} &
    \raisebox{-.5\height}{\includegraphics[height = .69in]{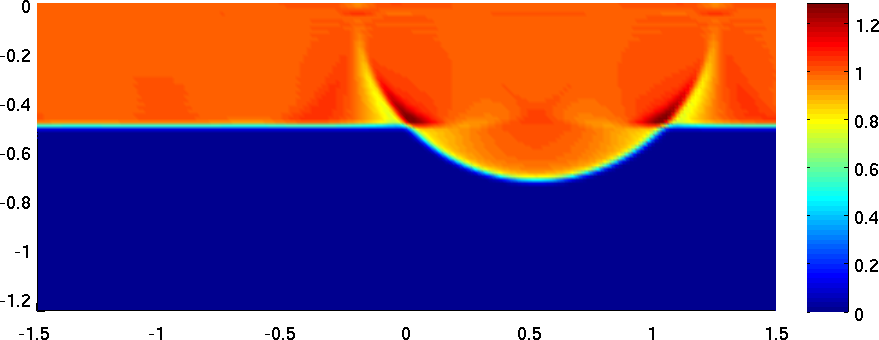}} \\
    $\tau_{4,2}$ &
    \raisebox{-.5\height}{\includegraphics[height = .69in]{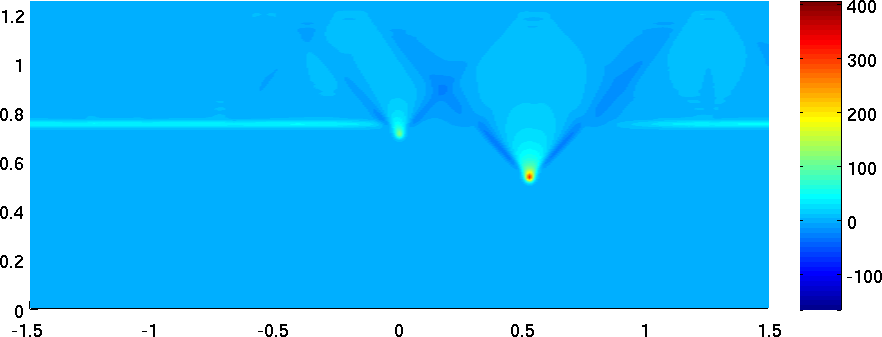}} &
    \raisebox{-.5\height}{\includegraphics[height = .69in]{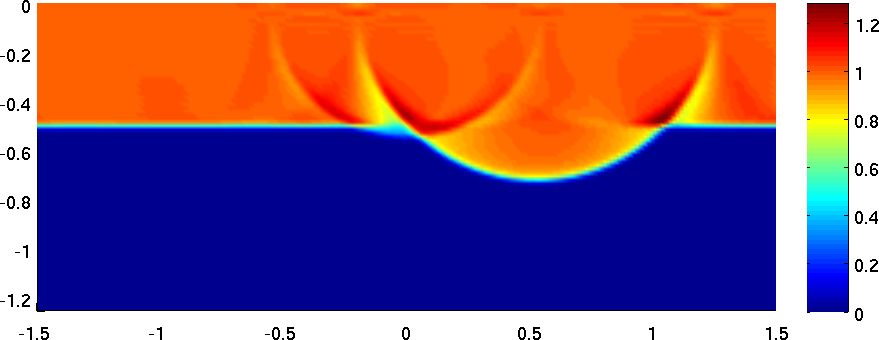}} \\
    \midrule
    $\tau_{3,3}$ &
    \raisebox{-.5\height}{\includegraphics[height = .69in]{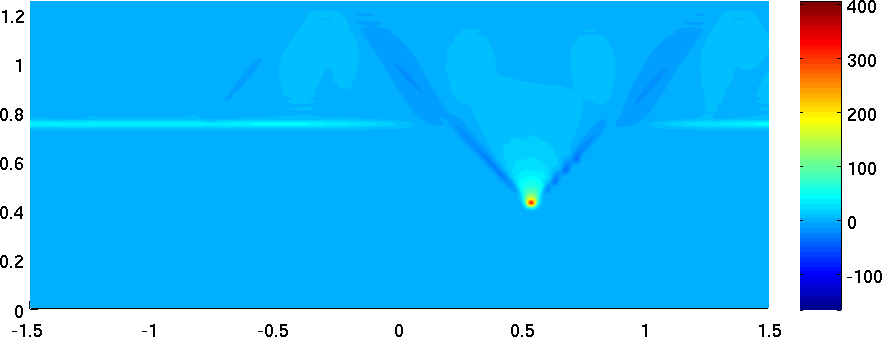}} &
    \raisebox{-.5\height}{\includegraphics[height = .69in]{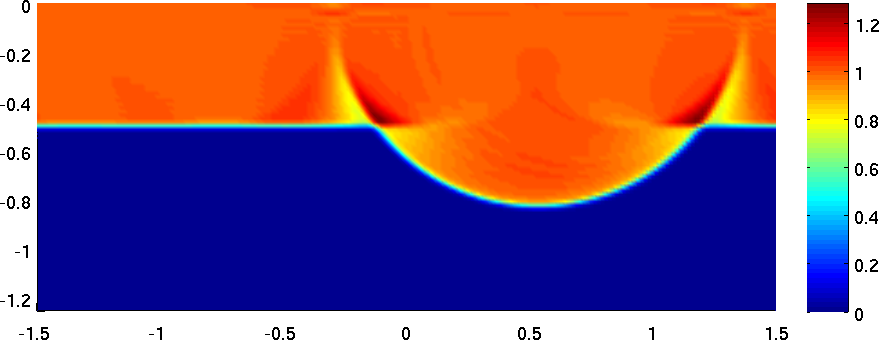}} \\
    $\tau_{4,3}$ &
    \raisebox{-.5\height}{\includegraphics[height = .69in]{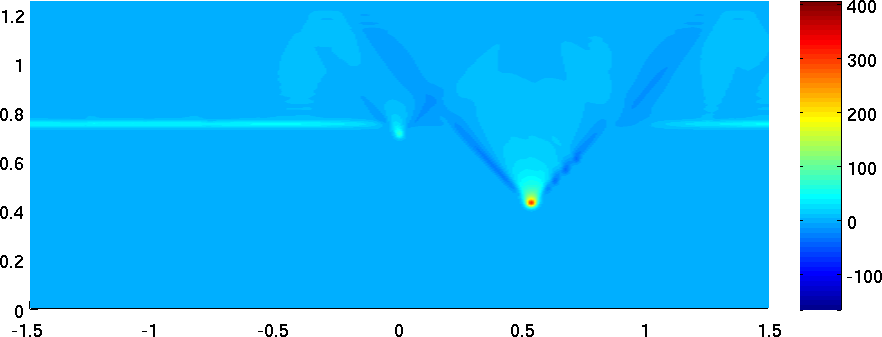}} &
    \raisebox{-.5\height}{\includegraphics[height = .69in]{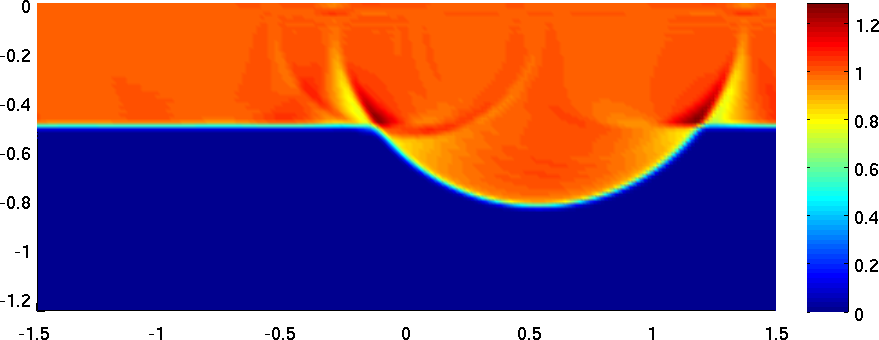}} \\
    \bottomrule
   \end{tabular}
  \caption{
    \label{tbl:SourcesAndWavefields}
    Illustration of the essential features of sources and waves used
    in distance estimation procedure. Solutions, $f_\alpha$, to the
    discretized control problem plotted next to their associated
    wavefields $u^{f_\alpha}(T,\cdot)$ approximating
    $1_{M(\tau)}$. Plotted for $\tau$ of the form $\tau_1 =
    s1_\Gamma$, $\tau_2 = \tau_{y}^{s+h}$, $\tau_{3,j} =
    \tau_{z}^{s+r_j}$, and $\tau_{4,j} = \tau_1 \vee \tau_2 \vee
    \tau_{3,j}$, where $y = (0,0)$, $z = (0.529,0)$, $h = .05$, and
    $r_j = 0.118, 0.235, 0.338$.}
\end{table}

We estimate distances between $z \in \Gamma$ and points of the form
$x(y,s)$ where $y = (0,0)$. In particular, for each fixed $s$ we
estimate the distances $d(x(y,s),(z_i,0))$ for uniformly spaced
$(z_i,0) \in [-1.175,1.175] \times \{0\} \subset \Gamma$. We take the
offset $\Delta z$ between the points $z_i$ equal to $\Delta z = 4
\Delta x_s = .0588$, and select the points $(z_i,0)$ to coincide with
every fourth source position.  As a proxy for estimating the distance
to $x(y,s)$, we use a target wave cap of the form
$\wavecap_\Gamma(y,s,h)$ with height $h = .025$, and estimate the
distances $r_{x(y,s)}((z_i,0))$ for $s = .125, .25, .375, .5$.


For each $s$ we solve the discrete control problem
(\ref{eqn:discreteControlProb}) in order to obtain estimates
$\hat{m}(s1_\Gamma)$ and $\hat{m}(\tau_{y}^{s+h})$ for the respective
volumes of $M(\Gamma,s)$ and $M(y,s+h)$. From these, we estimate the
volume of the target cap by,
\begin{equation*}
  \hat{m}_{\text{target cap}} = \hat{m}(\tau_y^{s+h}) - \hat{m}(s1_\Gamma).
\end{equation*}
For each point $(z_i,0)$, we also solve control problems to obtain
volume estimates $\hat{m}(\tau_{(z_i,0)}^{r_j + s})$ and
$\hat{m}(\tau_{(z_i,0)}^{r_j + s} \vee \tau_{y}^{s + h} \vee s
1_\Gamma)$, where we select the parameters $r_j, j = 1,\ldots, N_r$ so
that the two sets $\{s + r_j : j = 1,\ldots,N_r\} = \{t_{s,k} :
t_{s,k} > r \}$ coincide. We implement the distance estimation
procedure described in Lemma \ref{lemma:alternateDistanceProxy} to
estimate $r_{x(y,s)}((z_i,0))$ as follows: for each $r_j$ we estimate
the volume of $\wavecap_\Gamma(y,s,h) \cap
\wavecap_\Gamma((z_i,0),s,r_j)$ by first computing,
\begin{equation*}
  \hat{m}_{\text{overlap},j} = \hat{m}(\tau_{(z_i,0)}^{r_j + s} \vee \tau_{y}^{s +
    h} \vee s 1_\Gamma) - \hat{m}(\tau_{(z_i,0)}^{r_j + s}) -
  \hat{m}(\tau_{y}^{h + s}) + \hat{m}(s 1_\Gamma).
\end{equation*}
We then find the indices $j,j+1$ for which
\begin{equation}
  \label{eqn:rEstimateVolCriterion}
  \hat{m}_{\text{overlap},j} \leq \frac{1}{2}\hat{m}_{\text{target cap}} \leq \hat{m}_{\text{overlap},j+1},
\end{equation}
and estimate $r_h$ by linearly interpolating between $r_j$ and
$r_{j+1}$. This procedure approximates
(\ref{eqn:alternateDistanceProxy}). We depict the results of the
volume overlap estimation in Table \ref{tbl:relOverlapVol}. Since the
volumes in these images have all been normalized by the target cap
volumes, computing $r_h$ by (\ref{eqn:rEstimateVolCriterion})
corresponds to finding the $x$-value where the curve connecting the
data points passes through the line $y = 0.5$. We depict the distance
estimation results in Figure \ref{img:TestResults}.

\begin{table} 
  \begin{tabular}{c c c}
    \includegraphics[height = 1.4in]{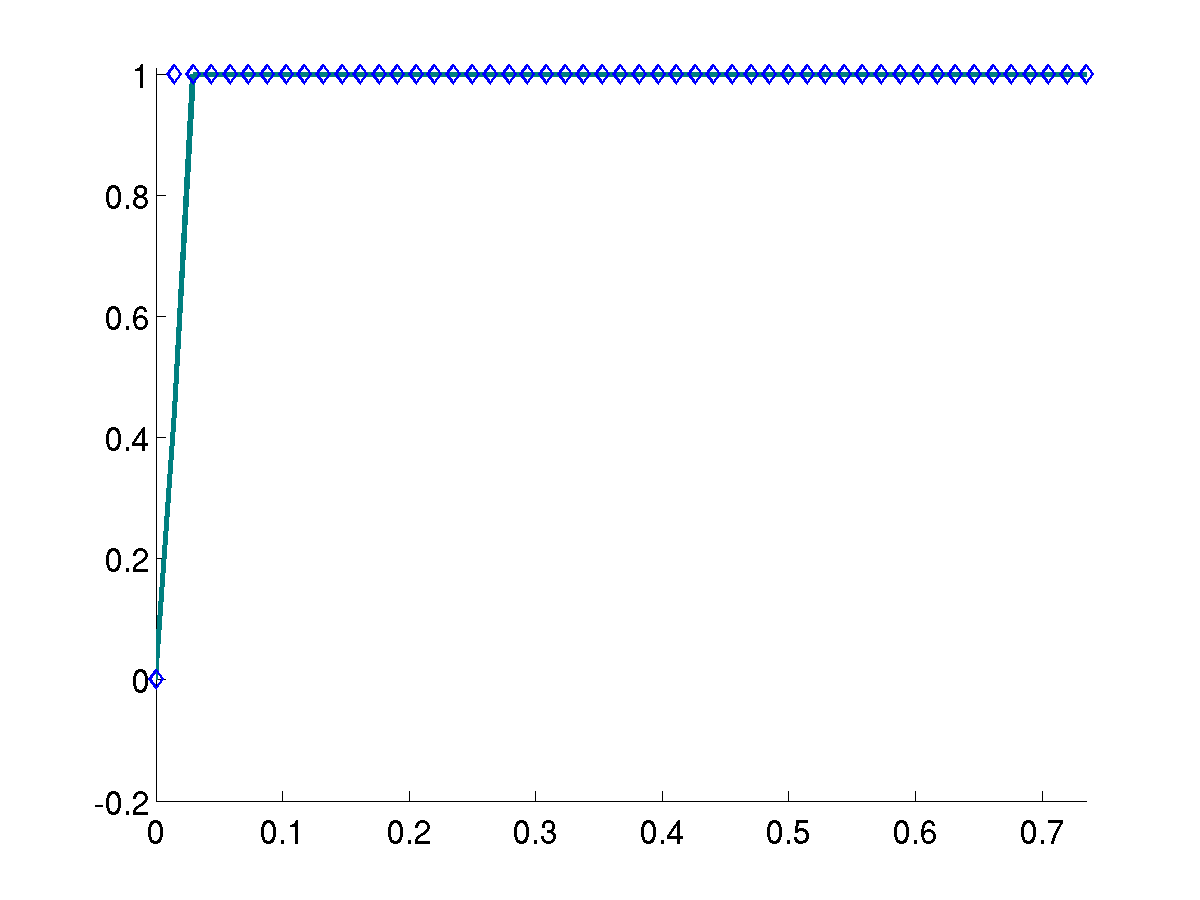} &
    \includegraphics[height = 1.4in]{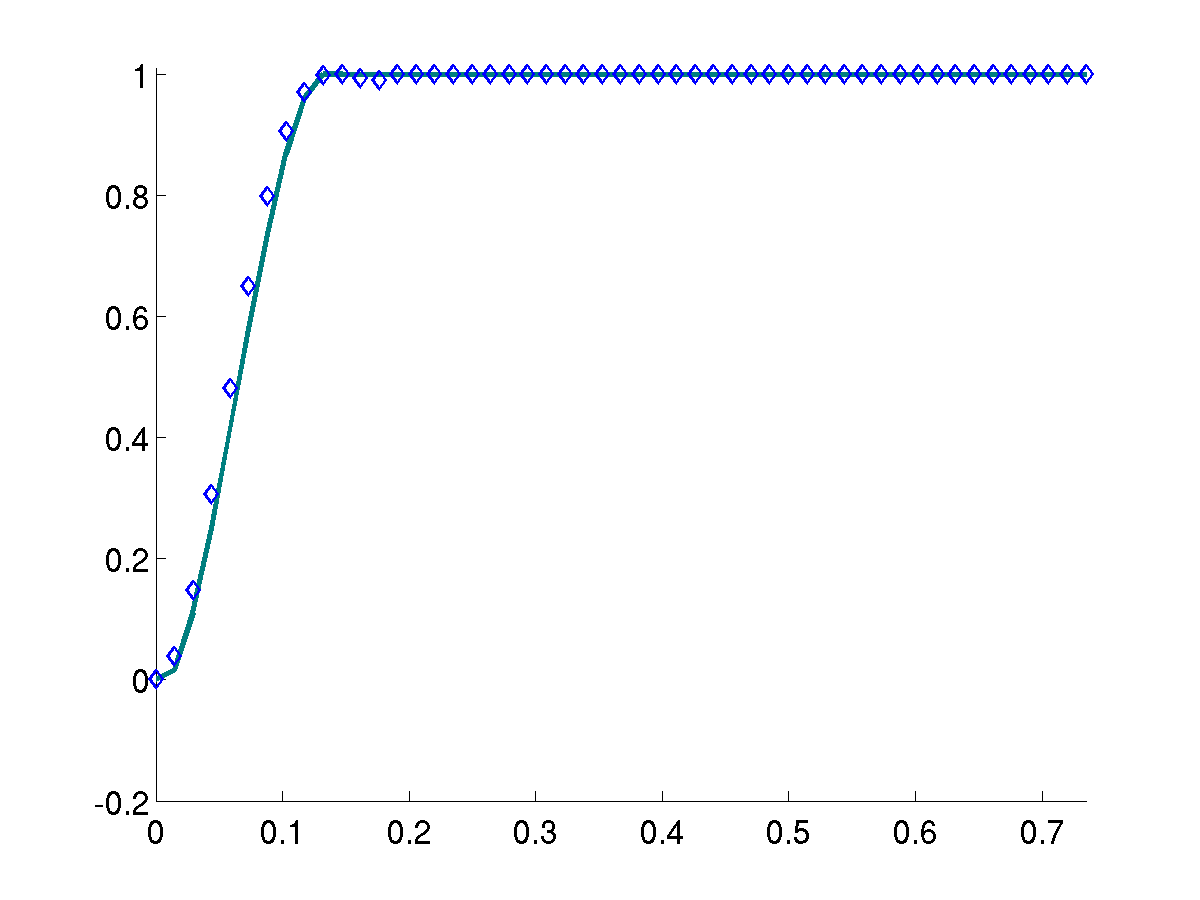} \\
    \includegraphics[height = 1.4in]{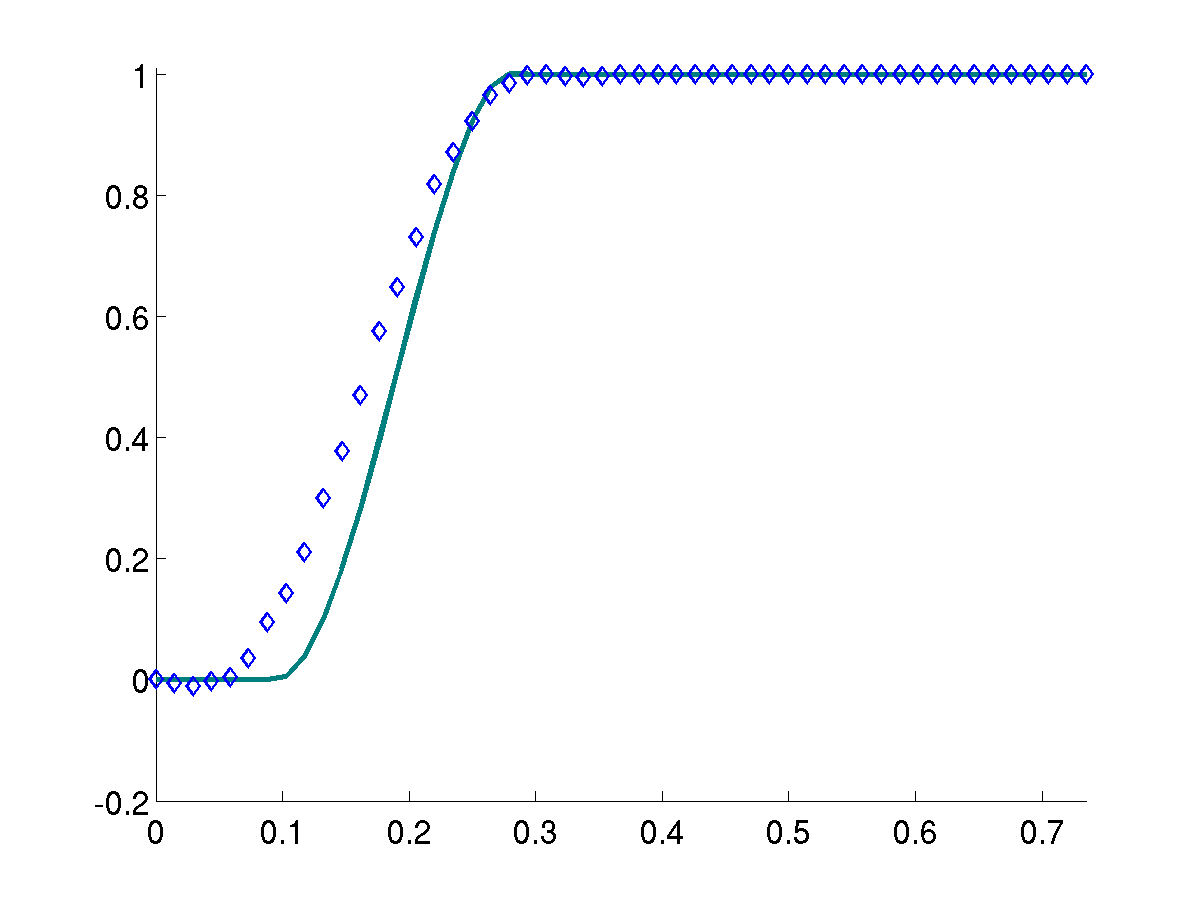} &
    \includegraphics[height = 1.4in]{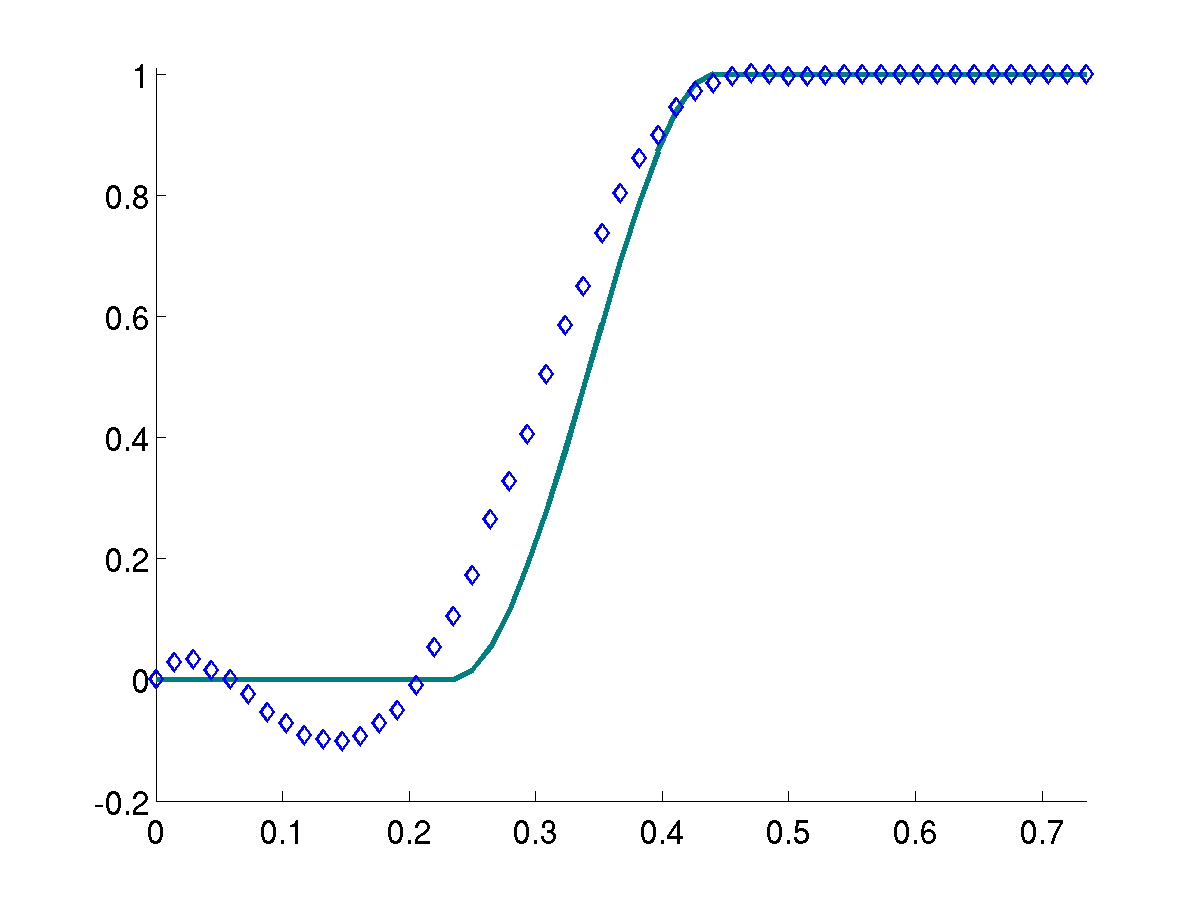} 
    \end{tabular}
  \caption{ 
    \label{tbl:relOverlapVol}
    Plots of relative overlap volumes, $\hat{m}_{\text{overlap}}/
    \hat{m}_{\text{target cap}}$, vs. $r$. Plotted for $s = .25$, $h =
    .025$ and (clockwise from the top left) $z = 0.176 \cdot i$, $i =
    0,\ldots,3$. The markers denote the relative overlap volumes
    estimated in the distance determination procedure, and the lines
    indicate the analytical relative overlap volumes. }
\end{table}


\begin{figure}
  \centering 
    \includegraphics[height = 3in]{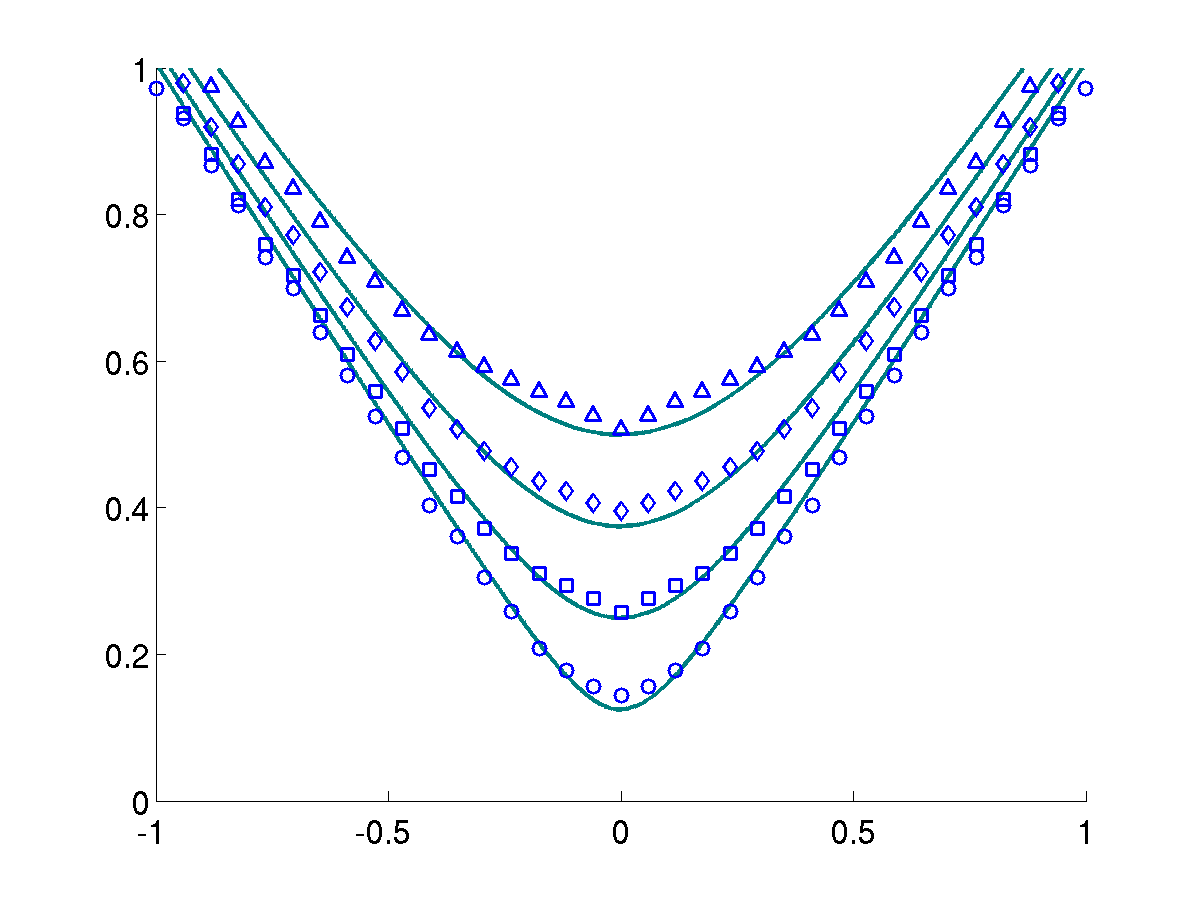}\\
    \caption{
      \label{img:TestResults}
      Distance estimates (markers) for $d(x(y,s),(z_i,0))$ for $y = (0,0)$ and $s =
      .125,.25,.375,.5$, plotted along with the true distances (solid curves). }
\end{figure}

\subsection{Discussion of sources of numerical errors and instability}
\label{subsec:Instability}

Examining Figure \ref{img:TestResults}, one can see that in each of
the estimated distance curves, the distances are over-estimated for $z
= (z_i,0)$ near $y = (0,0)$. This error results in part from the
distance estimation method. For example, when $z = y$ the correct
distance $d = r + s$ would be obtained by taking $r = 0$. On the other
hand, when $z = y$, both of the wave caps used in the distance
estimation procedure are centered on the same point, so for $0 \leq r
\leq h$ the variable wave cap, $\wavecap_\Gamma(z,s,r)$, coincides
with $\wavecap_\Gamma(z,s,r) \cap \wavecap_\Gamma(y,s,h)$. From the
definition of $r_h$, we find that we will have $0 < r_h < h$. Thus the
distance estimate $d_h$ will necessarily over-estimate
$d(y,x(y,s))$. Similar remarks apply for estimating
$d((z_i,0),x(y,s))$ for $(z_i,0)$ near $y$, although the strength of
this effect decreases as $(z_i,0)$ gets further from $y$. We call this
source of error \emph{geometric distortion}, since it results entirely
from the geometry of our distance estimation procedure and is
independent of errors arising from the control problems.  In Figure
\ref{img:GeometricDistortion} we depict the geometric distortion by
repeating our distance estimation technique with exact volume
measurements. Note that the distances in Figure
\ref{img:GeometricDistortion} are overestimated at all points, which
contrasts most with the distances estimated at large offsets in Figure
\ref{img:TestResults}.

\begin{figure}[h!]
  \centering 
  \includegraphics[height = 3in]{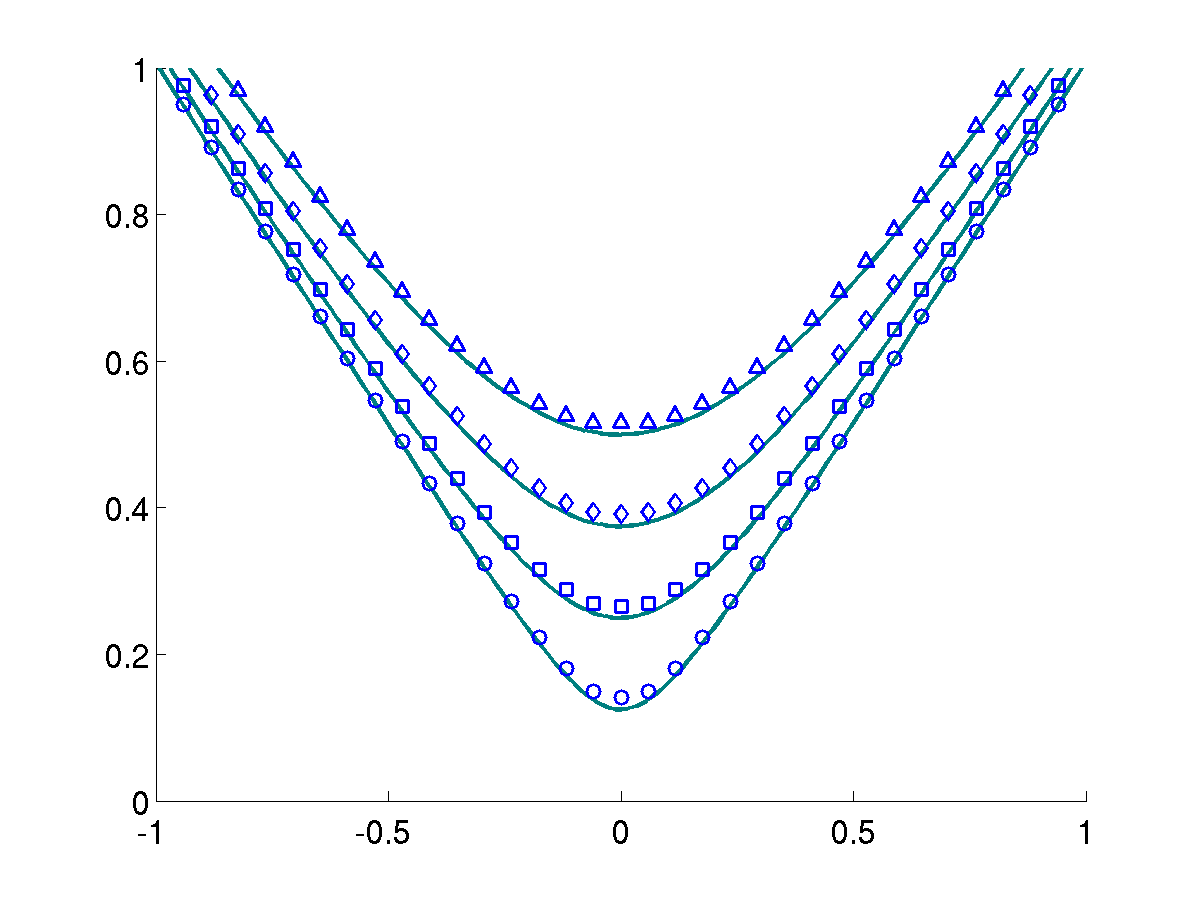}\\
  \caption{
    \label{img:GeometricDistortion}
    Demonstration of geometric distortion. Distances (markers) are
    estimated by using the distance estimation technique on exact
    volumes and plotted for $d(x(y,s),(z_i,0))$ for $y = (0,0)$ and $s
    = .125,.25,.375,.5$, along with the true distances (solid
    curves).}
\end{figure}

Our numerical tests suggest that the dominant source of error comes
from the control step. In order to discuss this instability, we return
to considering the continuum problem.  Taking $\tau \in
C(\overline{\Gamma})$, we can ask whether there exists $f \in
H^s(S_\tau)$ for some $s \in \mathbb{R}$ for which $W_{\tau} f =
1_{M(\tau)}$. This question can be answered by considering the more
general problem of exact controllability, in which one seeks to
determine when the equation $(u^f(T,\cdot), \p_t u^f(T,\cdot)) = (w_0,
w_1)$ has a solution in $H^s(S_\tau)$ for any $(w_0,w_1)$ belonging to
an appropriate space of Cauchy data for the wave equation.

In \cite{Bardos1992}, the question of exact controllability is
considered. One of the main results of that paper is that the ray
geometry of the wave equation can be used to determine necessary and
sufficient conditions for exact controllability. Using the same set of
ideas to those in \cite{Bardos1992}, it is shown in \cite{Bardos1996}
that in order for exact controllability to hold for $W_\tau$ in
$M(\tau)$ from $S_\tau$, the following \emph{geometric controllability
  condition} must hold:

\emph{ Each generalized bicharacteristic $(x(t),t)$ satisfying $x(T)
  \in M(\tau)$, passes over $S_\tau \cup S_\tau'$ in a non-diffractive
  point.  }

Here, $S_\tau' = \{(t,x) \in \Gamma \times (T,2T) : T \leq t \leq T +
\tau(x)\}$. We recall that $S_\tau$ is defined by
(\ref{eqn:Define_S_tau}), and note that $S_\tau'$ is the temporal
reflection of $S_\tau$ across $t = T$.  For a generalized
bicharacteristic $(x(t), t)$, the path $x(t)$ is a unit speed geodesic
in the interior of M and it is reflected according to Snell's law when
it intersects the boundary $\partial M$ transversally.  Tangential
intersections with the boundary can cause the path to glide along the
boundary, and in the case of an infinite-order contact, the path
$x(t)$ can be continued in many ways, see \cite{Bardos1992}.  We refer
also to \cite{Bardos1992} for the definition of non-diffractive
points. The geometric controllability condition is necessary for exact
control to hold from $S_\tau$, since when it fails for $(x,\xi) \in
S^*(M(\tau))$, propagation of singularities implies that for any $s
\in \R$ and any $f \in H^s(S_\tau)$, $(x,\xi) \not \in
\WF(u^f(T,\cdot))$, see e.g. \cite[Section 23]{Hormander1985}. Here,
$\WF(u^f(T,\cdot))$ denotes the wave front set of $u^f(T,\cdot)$, and
we refer to \cite[Def.  8.1.2]{Hormander1990} for its definition. We
have also provided the definition of wave front set in Appendix
\ref{appendix:WFComputation}. Thus, if $w \in L^2(M(\tau))$ has
$(x,\xi) \in \WF(w)$ then, for each $s \in \R$, there does not exist
$f \in H^s(S_\tau)$ for which $W_\tau f = w$.

In our computational experiment, the geometric controllability
condition actually fails over every point in $M(\tau)$. This is due to
the fact that at each $x \in M(\tau)$ there exists a family of
unit-speed geodesic rays with $(\gamma(T),\dot{\gamma}(T)) = (x,\xi)
\in S_x(M)$ and $\xi$ belonging to a cone over $x$, for which the
corresponding geodesics $\gamma$ fail to pass over $S_\tau \cup
S_\tau'$. In our computational experiment, we observe instabilities
near those $x \in M(\tau)$ where $\WF(1_{M(\tau)})$ meets the cone
over which exact control fails. In the case of $\tau = \tau_{y}^{s+h}
\vee s 1_\Gamma$, these effects occur where $\p M(\tau)$ fails to be
$C^{\infty}$ smooth and $\{x\} \times (\R^n \setminus 0) \subset
\WF(1_{M(\tau)})$ . We refer to Appendix \ref{appendix:WFComputation}
for further analysis. In particular these effects occur for $x$ in the
bottom left and right edges of $\wavecap_\Gamma(y,s,h)$, where $\p
M(\tau)$ fails even to be $C^1$. In addition, we similarly observe
instabilities near the points $(\pm L, -s)$, where the flat portion of
$\p M(\tau)$ transitions into a circle and fails to be $C^2$.

We demonstrate these effects in Figure \ref{fig:ut_with_cap}, by
plotting a wavefield $u^f(T,\cdot)$ approximating $1_{M(\tau)}$ for
$y = (0,0), s = .5$, and $h = .2$. The former instabilities occur near
the points $(\pm .5,-.5)$, and the latter instabilities occur near
$(\pm 1.15, -.5)$. We contrast this with the case where $\tau =
\tau_{y}^{s+h}$, which we show in Figure \ref{fig:ut_disk}. In this
second example, the domain of influence is a disk and every co-vector
in $\WF(1_{M(\tau)})$ can be controlled, and unlike the first example,
we observe no instabilities.  Note that in all of the examples in
Figure \ref{img:cornerConvergence} we use a smaller $\Gamma$ than in
our distance calculations, using $L = 1.153$ and a finer basis with
$a_t = a_x = 16 \cdot 10^3$.


\floatsetup[figure]{subcapbesideposition=center}

\begin{figure}[htb!]
  \centering
  
  \sidesubfloat[]{
    \includegraphics[width=4.0in]{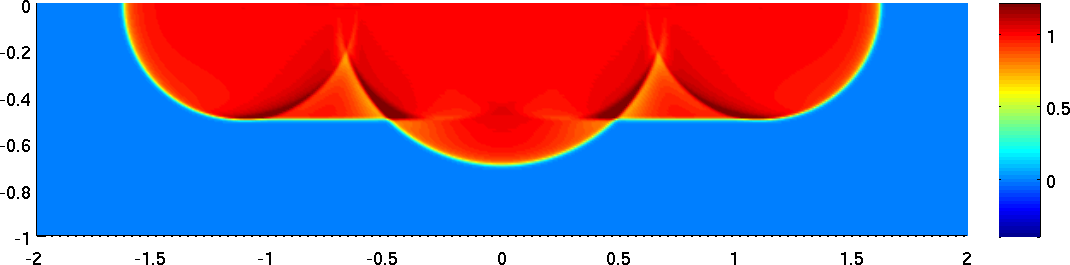} 
    \label{fig:ut_with_cap}
  }
  \hfill

  \sidesubfloat[]{
    \includegraphics[width=4.0in]{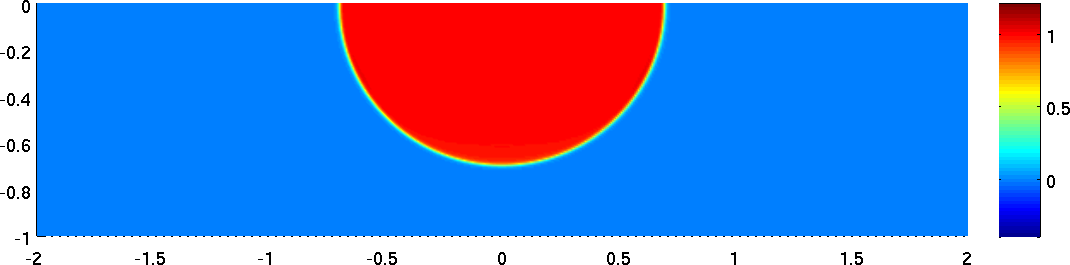} 
    \label{fig:ut_disk}
  }
  \hfill

  \sidesubfloat[]{    
    \includegraphics[width=4.0in]{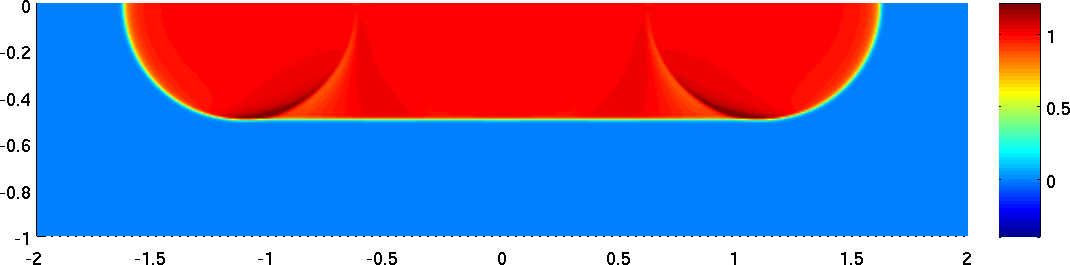}
    \label{fig:ut_without_cap}
  }
  \hfill

  \sidesubfloat[]{
    \includegraphics[width=4.0in]{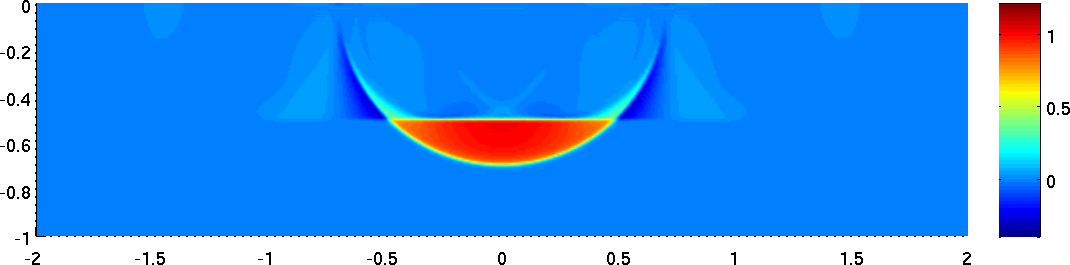}  
    \label{fig:ut_diff}
  }
  \hfill

  \caption{\label{img:cornerConvergence} (a) wavefield demonstrating
    instability of the solution to the control problem when
    $\WF(1_{M(\tau)})$ contains uncontrollable directions over $(\pm
    .5,-.5)$ and $(\pm 1.15, -.5)$ (b) A wavefield for which all
    directions in $\WF(1_{M(\tau)})$ are controlled. (c) Another
    wavefield demonstrating instability, with uncontrollable
    directions in $\WF(1_{M(\tau)})$ over $(\pm 1.15, -.5)$. (d) The
    difference between the wavefields in (a) and (c), note that this
    corresponds to an approximation to $1_{M(\wavecap(y,s,h))}$ as
    used in the distance estimation procedure. Moreover, the
    instabilities in (a) and (c) located over $(\pm 1.15, -.5)$ cancel
    each other.}
\end{figure}

In Figure \ref{fig:ut_without_cap} we plot the wavefield $u^f(T,\cdot)$
that approximates $1_{M(s1_\Gamma)}$. Note that as in the case of
$\tau = \tau_y^{s+h} \vee s1_\Gamma$, we observe instabilities near
the points $(\pm 1.15, -.5)$. In Figure \ref{fig:ut_diff} we plot the
difference between the wave fields approximating
$1_{M(\tau_y^{s+h}\vee s1_\Gamma)}$ and $1_{M(s1_\Gamma)}$, and note
that this difference yields an approximation to the characteristic
function of $\wavecap_\Gamma(y,s,h)$. In particular, notice that the
instabilities observed near $(\pm 1.15, -.5)$ in Figures
\ref{fig:ut_with_cap} and \ref{fig:ut_without_cap} completely cancel
in Figure \ref{fig:ut_diff}. Since our distance determination relies
primarily on the volumes of wave caps, which are obtained by taking
differences in this fashion, we find that the instabilities near the
cap bases tend to provide the main source of error for our distance
estimation procedure.

\section{Conclusions}

In this paper we have demonstrated a method to construct distances
between boundary points and interior points with fixed semi-geodesic
coordinates. The procedure is local in that it utilizes the local
Neumann-to-Dirichlet map for an acoustic wave equation on a Riemannian
manifold with boundary. Our procedure differs from earlier results in
that it utilizes volume computations derived from local data in order
to construct distances.  Finally, we have provided a computational
experiment to demonstrate the implementation of our distance
determination procedure and we have discussed the main sources of
numerical error.

\Appendix

\section{Wave front set of $1_{M(\tau)}$}
\label{appendix:WFComputation}

In Section~\ref{sec:Numerics} all of the functions $\tau$ that we
consider give rise to sets $M(\tau)$ with piecewise smooth
boundary. For all such $\tau$, if $x \in \p M(\tau)$ is a point for
which $\p M(\tau)$ is not smooth at $x$, then either $\p M(\tau)$ is
not $C^1$ at $x$, or $\p M(\tau)$ fails to be $C^2$ at $x$. In this
section, we examine the former case, by computing the wave front set
of $1_{M(\tau)}$ over a point $x$ where $\p M(\tau)$ fails to be
$C^1$. The other case is similar and we omit it. In particular, we
will show the following lemma:

\begin{lemma}
  \label{lemma:WFofCharFunc}
  Let $h: \R \to \R$ be continuous, and suppose that $h$ is smooth on
  $(-\infty, 0)$ and on $(0, \infty)$, and that $h$ does not belong to
  $C^1(\R)$.  Let $A$ be the set, $A = \{(x^1, x^2) \in \R^2:\ x^2 \le
  h(x^1)\}.$ Then, $\{0\} \times (\R^2 \setminus 0) \subset \WF(1_A)$.
\end{lemma}


From the lemma, it follows that if $\tau$ is one of the functions
considered in Section~\ref{sec:Numerics} and $\p M(\tau)$ fails to be
$C^1$ at $x \in \p M(\tau)$, then $\{x\} \times (\R^2 \setminus 0) \subset
\WF(1_{M(\tau)})$. Put differently, $\WF(1_{M(\tau)})$ contains all
cotangent directions above the point $x$.

Before proceeding with the proof, we recall the definition of wave
front set,
\begin{definition}
  Let $X \subset R^n$ open. If $u \in \mathcal{D}'(X)$, then the
  closed subset of $X \times (\R^n \setminus 0)$ defined by,
  \begin{equation*}
    \WF(u) = \{ (x,\xi) \in X \times (\R^n \setminus 0) : \xi \in
    \Sigma_x(u) \}
  \end{equation*}
  is called the \emph{wave front} set of $u$, where $\Sigma_x(u)$ is the set
  \begin{equation*}
    \Sigma_x(u) = \bigcap_{\phi} \Sigma(\phi u),\quad \phi \in
    C_0^\infty(X),~\phi(x) \neq 0,
  \end{equation*}
  and, for $v \in \mathcal{D}'(X)$, $\Sigma(v)$ is the complement of
  the set of $\eta$ in $\R^n \setminus 0$ for which there exists a
  conic neighborhood $V$ of $\eta$ such that for each $N \in
  \mathbb{N}$ there exist $C_N > 0$ such that, for all $\xi \in V$,
  \begin{equation*}
    |\hat{v}(\xi)| \leq C_N (1 + |\xi|)^{-N}.
  \end{equation*}
\end{definition}

To prove Lemma \ref{lemma:WFofCharFunc}, we require the following
result,

\begin{lemma}
Let $h \in \R$, $u \in C_0^\infty(\R)$. Suppose that $\phi \in
C^\infty(\R)$ is real valued and has no critical points in $\supp(u)$.
Then
\begin{align*}
&\int_{-\infty}^{h} u(x) e^{-i\lambda \phi(x)} dx
= \frac {i  u(h) e^{-i\lambda \phi(h)}}{\lambda \phi'(h)}
+  \frac {\left. \p_x( u /\phi') \right|_{x=h} e^{-i\lambda \phi(h)}}{\lambda^2 \phi'(h)} 
+ \frac {R(\lambda,h,\phi,u)} {\lambda^3},
\end{align*}
where $|R|$ is bounded by a constant that depends only on $\supp(u)$,
$\min_{x \in \supp(u)} |\phi'(x)|$ and the $C^3$ norms of $u$ and $\phi$.
\end{lemma}

{\em Proof}.
We define the differential operator $L = i\phi'^{-1}\p_{x}$. Then
\begin{align*}
&\int_{-\infty}^{h} u e^{-i\lambda \phi} dx
= \frac 1 \lambda \int_{-\infty}^{h} u L e^{-i\lambda \phi} dx
\\&\quad
= \left. \frac {i u e^{-i\lambda \phi} }{\lambda \phi'} \right|_{x=h}
+ \frac 1 \lambda \int_{-\infty}^{h} (L^* u) e^{-i\lambda \phi} dx.
\end{align*}
Notice, $L^*u = -i \p_x( u / \phi')$. Then,
\begin{align*}
&\frac 1 \lambda \int_{-\infty}^{h} (L^* u) e^{-i\lambda \phi} dx
= \frac 1 {\lambda^2} \int_{-\infty}^{h} (L^*u) L e^{-i\lambda \phi} dx
\\&\quad
= \left. \frac {\p_x( u /\phi') e^{-i\lambda \phi}}{\lambda^2 \phi'}\right|_{x=h}
+ \frac 1 {\lambda^2} \int_{-\infty}^{h} ((L^*)^2 u) e^{-i\lambda \phi} dx.
\end{align*}
Finally,
\begin{align*}
&\frac{1}{\lambda^2} \int_{-\infty}^{h} ((L^*)^2 u) e^{-i\lambda \phi} dx
= \frac 1 {\lambda^3} \int_{-\infty}^{h} ((L^*)^2u) L e^{-i\lambda \phi} dx
\\&\quad
= \left. \frac {i ((L^*)^2 u) e^{-i\lambda \phi}}{\lambda^3 \phi'}\right|_{x=h}
+ \frac 1 {\lambda^3} \int_{-\infty}^{h} ((L^*)^3 u) e^{-i\lambda \phi} dx. \qquad\endproof
\end{align*}

\begin{proof}[of Lemma \ref{lemma:WFofCharFunc}]
  Let $u \in C_0^\infty(\R^2)$ and suppose that $u = 1$ near the origin.
We consider the Fourier transform
\begin{align*}
\widehat{u 1_A}(\lambda \xi)
&= \int_{\R^2} u(x) 1_A(x) e^{-i\lambda \xi x} dx
\\&
= \int_{-\infty}^{+\infty} e^{-i\lambda \xi_1 x^1} \int_{-\infty}^{h(x^1)}
u(x) e^{-i\lambda \xi_2 x^2} dx^2 dx^1,
\end{align*}
where $\xi$ is a unit vector and $\lambda > 0$. 

Suppose first that $\xi_2 \ne 0$. Then 
\begin{align*}
&\int_{-\infty}^{h(x^1)} u(x) e^{-i\lambda \xi_2 x^2} dx^2
= \frac {i w(x^1) e^{-i\lambda \xi_2 h(x^1)}}{\lambda \xi_2}
+ \frac {v(x^1) e^{-i\lambda \xi_2 h(x^1)}}{\lambda^2 \xi_2^2}
+ \frac {R} {\lambda^3},
\end{align*}
where $w(x^1) = u(x^1, h(x^1))$ and $v(x^1) = \p_{x^2}u(x^1,
h(x^1))$.  Note that $R$ is compactly supported with respect to $x^1$
since $u$ is.  Note also that $w$ and $v$ are compactly supported and
$w = 1$ near the origin. If $\supp(u)$ is small then $\supp(w)$ and
$\supp(v)$ are also small. Moreover, if $\supp(w) \cup \supp(v)$ is
small enough then $h$ is smooth in $(\supp(w) \cup \supp(v)) \setminus
0$. In particular, $w$ and $v$ are smooth then.

We define $\phi(x^1) = \xi_1 x^1 + \xi_2 h(x^1)$, $\phi_\pm = \phi|_{\pm
  x^1 > 0}$, and define also $h_\pm$ analogously.  Suppose that
$\phi_-'(0) = \xi_1 + h_-'(0)\xi_2 \ne 0$. Then $\phi$ has no critical
points in the set $\{x^1 < 0\} \cap \supp(w)$ if $\supp(w)$ is small
enough. Therefore
\begin{align*}
\int_{-\infty}^0 w e^{-i\lambda \phi} dx^1
= \frac {i}{\lambda \phi_-'(0)}
+ \frac {R_+(\lambda,h,\phi,u)} {\lambda^2}.
\end{align*}
Likewise, if also $\phi_+'(0) \ne 0$, then
\begin{align*}
\int_0^{+\infty} w e^{-i\lambda \phi} dx^1
= \frac {-i}{\lambda \phi_+'(0)}
+ \frac {R_-(\lambda,h,\phi,u)} {\lambda^2}.
\end{align*}
So,
$$
\int_{-\infty}^{+\infty} w e^{-i\lambda \phi} dx^1
= i\left( \frac{1}{\phi_-'(0)} -\frac{1}{\phi_+'(0)} \right) \lambda^{-1} + \mathcal O(\lambda^{-2}).
$$ 
An analogous argument applies to $v$, showing that,
$$ \int_{-\infty}^{+\infty} v e^{-i\lambda \phi} dx^1 =
\mathcal{O}(\lambda^{-1}).$$
Hence, we conclude,
$$
\widehat{u 1_A}(\lambda \xi)
= \frac{1}{\xi_2}\left( \frac{1}{\phi_+'(0)} -\frac{1}{\phi_-'(0)} \right) \lambda^{-2} + \mathcal O(\lambda^{-3}).
$$
Thus $\widehat{u 1_A}$ does not decay rapidly if $\phi_+'(0) \ne \phi_-'(0)$
which again is equivalent to $h_+'(0) \ne h_-'(0)$.

To summarize, if $h_+'(0) \ne h_-'(0)$ then all the directions except
possibly $(1,0)$ and the four directions
$$ (-h_\pm'(0) \xi^2, \xi^2), \quad \text{where} \quad |\xi_2| =
(|h_\pm'(0)|^2 + 1)^{-1/2},$$ are in $\Sigma_0(1_A)$. Finally, since
$\Sigma_0(1_A)$, is a closed conic subset of $\R^2 \setminus 0$, we
conclude that $\Sigma_0(1_A) = \R^2 \setminus 0$, and hence $\{0\}
\times (\R^2 \setminus 0) \subset \WF(1_A)$. \qquad
\end{proof}


\bibliographystyle{siam} 
\bibliography{Bibliography}

\end{document}